\documentclass[11pt,reqno]{amsart}

\usepackage{amsmath,amsthm,amscd,amssymb,amsfonts, amsbsy}
\usepackage{appendix}
\usepackage{latexsym}
\usepackage{txfonts}
\usepackage{exscale}
\usepackage{mathrsfs}
\usepackage{esint} 
\usepackage{enumitem} 
\usepackage{color}
\usepackage[perpage]{footmisc}
\usepackage[normalem]{ulem} 

\usepackage{relsize}
\usepackage{tikz}
\usepackage{bbm} 

\usepackage{graphicx} 

\def\YYint#1#2#3{{\setbox0=\hbox{$#1{#2#3}{\iint}$}
    \vcenter{\hbox{$#2#3$}}\kern-.51\wd0}}
 
\usepackage[
  hmarginratio={1:1},     
  vmarginratio={1:1},     
  textwidth=400pt,			  
	textheight=580pt,        
  heightrounded,          
]{geometry}

\usepackage[utf8]{inputenc}

\usepackage[
	colorlinks=true,
	citecolor=blue,
	linktoc=all,
	pagebackref,
	hypertexnames=false
]{hyperref}

\usepackage{appendix}

\definecolor{br}{rgb}{1, 0.4,0}

\allowdisplaybreaks

\numberwithin{equation}{section}

\theoremstyle{plain}
\newtheorem{theorem}[equation]{Theorem}

\newtheorem{lemma}[equation]{Lemma}

\theoremstyle{definition}
\newtheorem{defn}[equation]{Definition}

\theoremstyle{remark}
\newtheorem{remarks}[equation]{Remarks}

\numberwithin{equation}{section}

\newcommand{\RR}{{\mathbb{R}}}
\newcommand{\NN}{{\mathbb{N}}}

\renewcommand{\emptyset}{\mbox{\textup{\O}}}

\DeclareMathOperator{\dist}{dist}

\DeclareMathOperator{\Real}{Re}
\DeclareMathOperator{\Imag}{Im}

\DeclareMathOperator{\graph}{graph}

\newcommand{\St}{\mathcal{S}}

\reversemarginpar

\begin{document}

\allowdisplaybreaks

\title{Examples of non-Dini domains with large singular sets}

\author{Carlos Kenig}
\address{Carlos Kenig
\\ 
Department of Mathematics
\\
University of Chicago
\\
Chicago, IL 60637, USA}
\email{cek@math.uchicago.edu}

\author{Zihui Zhao}
\address{Zihui Zhao
\\ 
Department of Mathematics
\\
University of Chicago
\\
Chicago, IL 60637, USA}
\email{zhaozh@uchicago.edu}

\thanks{The first author was supported in part by NSF grants DMS-1800082 and DMS-2153794, and the second author was partially supported by NSF grant DMS-1902756.}

\vspace{2em}

\dedicatory{\vskip.5em Dedicated to David Jerison on his 70'th birthday}


\begin{abstract}
	Let $u$ be a non-trivial harmonic function in a domain $D\subset \RR^d$ which vanishes on an open set of the boundary.
	In a recent paper, we showed that if $D$ is a $C^1$-Dini domain, then within the open set the singular set of $u$, defined as $\{X\in \overline{D}: u(X) = 0 = |\nabla u(X)| \} $, has finite $(d-2)$-dimensional Hausdorff measure. In this paper, we show that the assumption of $C^1$-Dini domains is sharp, by constructing a large class of non-Dini (but almost Dini) domains whose \textit{singular sets} have infinite $\mathcal{H}^{d-2}$-measures.
\end{abstract}

\maketitle

\section{Introduction}

We consider the following question, which is inspired by a classical question asked by Bers (see the introduction in \cite{KZ}):
\begin{equation}
  \tag{Q}\label{Q}
  \parbox{\dimexpr\linewidth-6em}{%
    \strut
    Suppose $u$ is a non-trivial harmonic function in a domain $D$, and that $u=0$ on a relatively open set of the boundary $B_{2R}(0) \cap \partial D$. How big can the singular set $\St:= \{X\in B_R(0) \cap \partial D: \nabla u(X) = 0\}$ be?
    \strut
  }
\end{equation}
When $D$ is a $C^{1,1}$ domain, Lin proved that $\St$ has zero $(d-1)$-dimensional Hausdorff measure, and that $\St$ is a $(d-2)$-dimensional set, see \cite[Theorem 2.3]{Lin}.
Adolfsson, Escauriaza and Kenig \cite{AEK} (see also Kenig-Wang \cite{KW} for an alternative proof) extended the result to convex domains. This was then followed by works of Adolfsson-Escauriaza \cite{AE} and Kukavica-Nystr\"om \cite{KN}, who proved (using different methods) the result for $C^1$-Dini domains (see Definition \ref{def:Dini}). 
 Recently, Tolsa \cite{Tol} proved that for all $C^1$ domains (or Lipschitz domains with sufficiently small Lipschitz constant), the set $\St$ has zero $(d-1)$-dimensional Hausdorff measure.
 
In a recent work, we proved the following theorem:

\begin{theorem}[\cite{KZ}]\label{thm:KZ}
	Let $D$ be a $C^1$-Dini domain in $\RR^d$ (see Definition \ref{def:Dini}) with $0\in \partial D$, and let $R>0$. Suppose $u$ is a non-trivial harmonic function in $D\cap B_{50R}(0)$, and that $u=0$ on $\partial D \cap B_{50R}(0)$. Then the singular set
	\[ \St(u) := \{X\in \overline{D\cap B_{50R}(0)}: u(X) = 0 = |\nabla u(X)| \} \]
	satisfies that $\St(u)\cap B_R(0)$ is $(d-2)$-rectifiable, and
	\[ \mathcal{H}^{d-2}(\St(u) \cap B_R(0)) \leq C <+\infty, \]
	where the constant $C$ depends on $d, R$ and (the upper bound of) the frequency function of $u$ centered at $0$ with radius $50R$.
\end{theorem}
In short, the theorem says that when $D$ is a $C^1$-Dini domain, the singular set at the interior and boundary, $\St(u) \cap B_R(0)$, is $(d-2)$-rectifiable, and its $(d-2)$-dimensional Hausdorff measure is finite. We remark that a similar result for convex domains can be found in \cite{McCurdy}.

It is natural to ask whether such fine estimate (i.e. $\mathcal{H}^{d-2}$ estimate) of the singular set can be extended to more general domains, for example Lipschitz domains with small constants. In that setting, recall Tolsa showed that $\St(u) \cap \partial D$ has surface measure zero in $B_R(0)$ (see \cite{Tol}). The answer is no in general, because if the domain is less regular than $C^1$-Dini, the gradient of harmonic functions which vanish on an open subset of the boundary may not exist everywhere in that open set, and thus it does not make sense to talk about its $\mathcal{H}^{d-2}$-measure. The goal of this paper is to give counterexamples which say that $C^1$-Dini domains are indeed the optimal class of domains for which Theorem \ref{thm:KZ} holds.
More precisely if $D$ is less regular than $C^1$-Dini and a harmonic function $u$ vanishes on an open subset of $\partial D$, in general we can not make sense of $\nabla u$ at the boundary. However, in the special case when $u$ is a \emph{non-negative} harmonic function in $D$ which vanishes in $\partial D \cap B_{2R}(0)$ with $0\in \partial D$, by the comparison principle (see Lemma \ref{lm:comparison}) $u$ is comparable to the Green's function $G$ in $D\cap B_R(0)$. Hence for $\sigma$-almost every $x\in \partial D \cap B_R(0)$ (where $\sigma:= \mathcal{H}^{d-1}|_{\partial D}$ denotes the boundary surface measure of $D$), we have
\[ \nabla u(x) \approx \partial_n G(x) \approx \frac{d\omega}{d\sigma}(x), \]
where $\partial_n G$ denotes the normal derivative of $G$ at the boundary, $d\omega/d\sigma$ denotes the Poisson kernel of the harmonic measure $\omega$ (whose pole is the same as that of the Green's function $G$). Since the upper and lower densities of the Radon measure $\omega$ are defined everywhere (and take values in $[0,+\infty]$), we can use the following set\footnote{Throughout the paper, we always use $\Delta_r(p)$ to denote a surface ball at the boundary, as in
\[ \Delta_r(p) := \partial D \cap B_r(p). \]}
\begin{align*}
	 \left\{p\in \partial D: \liminf_{r\to 0} \frac{\omega(\Delta_r(p))}{r} = \limsup_{r\to 0} \frac{\omega(\Delta_r(p))}{r} = 0 \right\}  
\end{align*} 
in place of the boundary singular set of $u$, namely $\{p\in \partial D: \nabla u(p) = 0\}$. Roughly speaking, we showed that when the domain $D$ barely fails to be $C^1$-Dini, the $\mathcal{H}^{d-2}$-measure of the above set could be infinite.

\begin{theorem}\label{thm:main}
	Given a monotone non-decreasing function $\theta: \RR_+ \to \RR_+$ which satisfies 
	\begin{equation}\label{cond:theta}
		\lim_{r\to 0+} \theta(r) = 0 \quad \text{ and } \quad \int_0^* \frac{\theta(r)}{r} \, dt = +\infty, 
	\end{equation} 
	there exist a $C^1$ function $\varphi: \RR\to \RR$ and a $C^1$ domain 
	\[ D := \{(x, t)\in \RR \times \RR: x\in \RR, t> \varphi(x) \} \]
	such that the following holds:
	\begin{itemize}
		\item there exists a bounded set $S\subset \RR$ containing infinitely (countably) many points, such that for each $x_0 \in S$ the modulus of continuity of $\nabla \varphi$ at $x_0$, denoted by $\alpha(r)$, satisfies
			\[ \theta(r) \leq \alpha(r) \leq \theta(4r); \]
		\item $\varphi \in C^2(\RR\setminus \overline{S})$;
		\item let $\omega$ denote the harmonic measure in $D$, we have that
			\[ \left\{p\in \partial D : \liminf_{r\to 0} \frac{\omega(\Delta_r(p))}{r} = \limsup_{r\to 0} \frac{\omega(\Delta_r(p))}{r} 
			= 0  \right\} \supset \graph_S \varphi, \]
			where $\graph_S \varphi := \{(x, \varphi(x): x\in S \}$.
	\end{itemize}
	In particular, the set 
	\[ \left\{p\in \partial D : \liminf_{r\to 0} \frac{\omega(\Delta_r(p))}{r} = \limsup_{r\to 0} \frac{\omega(\Delta_r(p))}{r} 
	= 0  \right\} \]
	is infinite.
\end{theorem} 
\begin{remarks}
	\begin{itemize}
		\item We can easily extend the above example of planar domains to domains in $\RR^d$, by considering the domains $\tilde{D}: = D \times \RR^{d-2} \subset \RR^d$. In this case, the \textit{singular} set in the boundary of $\tilde{D}$ is equal to $\graph_S \varphi \times \RR^{d-2}$ and has infinite $\mathcal{H}^{d-2}$-measure.
		\item In the proof we will also show that whenever $x\ll 0$ and $x\gg 0$, the function $\varphi(x)$ is linear, and thus the boundary $\partial D = \graph_{\RR} \varphi$ is flat when we are sufficiently far from the center $(0, \varphi(0))$. Therefore it is not hard to close off $D$ so it becomes a bounded domain, and at the same time maintain that $\partial D$ has $C^2$ regularity except for points in $\graph_S \varphi$. 
			Thus the result also holds for bounded domains.
	\end{itemize}	
\end{remarks}

We remark that when $D$ is a $C^1$-Dini domain and the harmonic function $u$ in consideration is non-negative, the Hopf maximum principle implies that 
\[ |\nabla u| = \partial_\nu u \geq c >0 \text{ at the boundary}, \]
where $\partial_\nu u$ denotes the normal derivative of $u$ with the normal vector $\nu$ pointing inwards, see \cite{CK}. (The Hopf maximum principle in \cite{CK} was proven for solutions to parabolic equations, but by taking a slice at a fixed positive time the elliptic analogue follows.) In particular, this implies that for $C^1$-Dini domains, the \textit{singular set}
\[ \left\{p\in \partial D: \liminf_{r\to 0} \frac{\omega(\Delta_r(p))}{r} = \limsup_{r\to 0} \frac{\omega(\Delta_r(p))}{r} 
= 0 \right\} = \emptyset. \]

In a related work \cite[Section 9]{Josep}, the author constructed the following example (credited to Tolsa): there exist Lipschitz domains $D\subset \RR^2$ with small constants such that the \textit{singular set} at the boundary 
\[ \left\{p\in \partial D: \lim_{r\to 0} \frac{\omega(\Delta_r(p))}{r} \text{ exists and is equal to } 0 \right\} \]
has Hausdorff dimension as close to $1$ as we want. In particular, it indicates that for Lipschitz domains, one can not expect a better answer to \eqref{Q} than saying that the singular set has zero surface measure. (For comparison, our examples show that in order to obtain the sharp $(d-2)$-dimensional estimate for \eqref{Q} as in \cite{KZ}, the assumption of $C^1$-Dini domains can not be weakened.)
 These Lipschitz domains are built by taking the union of cones with vertices at a fat Cantor set, whose Hausdorff dimension can be chosen sufficiently close to $1$. The purpose of their example is similar to ours, but the constructions are completely different. Besides, our example of domains are better than $C^1$-regular, instead of just Lipschitz, but the \textit{singular set} is $0$-dimensional (albeit infinite), rather than $(1-\epsilon)$-dimensional.

The proof of Theorem \ref{thm:main} is inspired by the work of the first named author \cite{Kenig}. There it was shown that one can construct a Lipschitz domain in $\RR^2$ with prescribed tangent vectors on its boundary, such that its harmonic measure is given by the exponential of the Hilbert transform of that prescribed function, see for example \cite[Lemma 1.11]{Kenig}.
The plan of the paper is as follows. We recall some definitions and preliminary results in Section \ref{sec:prelim}. Then to fix ideas, we first construct Lipschitz domains with the desired properties in Section \ref{sec:Lip}. These domains are explicit and not difficult to visualize. In Section \ref{sec:C1}, we construct the desired $C^1$ domains for every modulus of continuity $\theta$ satisfying \eqref{cond:theta}.
\bigskip

\textbf{Acknowledgement.} Remarks from C.K.: I have known David for 45 years, as a close collaborator and a dear friend. When our collaboration began, at the start of our careers, it brought us great excitement, and it greatly helped to launch my professional path. Our close friendship, through both happy times and very difficult ones, has been a joy. Thank you David! 
From Z.Z.: I have been greatly influenced and inspired by Prof. Jerison's work, especially how he makes connections between different fields of math. For that I am always grateful. Happy birthday, Prof. Jerison!

\section{Preliminaries}\label{sec:prelim}
\begin{defn}[Dini domains]\label{def:Dini}
	Let $\theta: \RR \to \RR$ be a non-decreasing function which satisfies $\lim_{r\to 0+} \theta(r) = 0$ and
	\[
		\int_0^* \frac{\theta(r)}{r} < \infty. 
	\]
	A connected domain $D$ in $\RR^d$ is a \textit{Dini domain} with parameter $\theta$ if for each point $X_0 \in \partial D$ there is a coordinate system $X=(x,x_d), x\in \RR^{d-1}, x_d\in \RR$ such that $X_0=(0,0)$ with respect to this coordinate system, and there are a ball $B$ centered at $X_0$ and a Lipschitz function $\varphi: \RR^{d-1} \to \RR$ verifying the following
	\begin{enumerate}
		\item $\|\nabla \varphi\|_{L^\infty(\RR^{d-1})} \leq C_0$ for some $C_0>0$;
		\item $|\nabla \varphi(x)- \nabla \varphi(y)| \leq \theta(|x-y|)$ for all $x,y \in \RR^{d-1}$;
		\item $D\cap B=\{(x,x_d)\in B: x_d > \varphi(x) \}$.
	\end{enumerate}
\end{defn}

The following integral will be used repeatedly in the computation of Hilbert transforms, so we state it as a lemma here.
\begin{lemma}\label{lm:prelim}
	Let $a<b$ be two real numbers. Suppose $x \notin [a,b]$, we have that
	\begin{equation}
		\int_a^b \frac{1}{x-y} \, dy = \log|x-a| - \log|x-b|.
	\end{equation}
\end{lemma}
\begin{proof}
	When $x<a<y$, let $z:=y-x>0$. By a change of variables, we have
	\[ \int_a^b \frac{1}{x-y} \, dy = \int_{a-x}^{b-x} -\frac{1}{z} \, dz = \log(a-x) - \log (b-x). \]
	When $x>b>y$, let $z=x-y>0$. By a change of variables, we have
	\[ \int_a^b \frac{1}{x-y} \, dy = \int_{x-a}^{x-b} \frac{1}{z} \, -dz = \int_{x-b }^{x-a} \frac{1}{z} \, dz = \log (x-a)-\log (x-b). \]
\end{proof}

We recall the following lemmas about positive solutions to elliptic PDEs (for a reference see \cite{CBMS}).
\begin{lemma}[Comparison principle]\label{lm:comparison}
	Let $D$ be a Lipschitz domain and $p\in \partial D$, $r>0$.
	Let $u, v\geq 0$ be two non-trivial harmonic functions in $D \cap B_{2r}(p)$ such that $u=v=0$ on $\Delta_{2r}(p) := \partial D \cap B_{2r}(p)$. Then for any $X\in D \cap B_r(p)$,
	\[ C^{-1} \, \frac{u}{v} (A_r(p)) \leq \frac{u}{v}(X) \leq C \, \frac{u}{v} (A_r(p)), \]
	where $C\geq 1$ is a universal constant, and $A_r(p)$ denotes the corkscrew point in $D\cap B_r(p)$, that is to say, there exists a constant $M>0$ only depending on the Lipschitz constant of the domain, such that $B_{r/M}(A_r(p)) \subset D\cap B_r(p)$.
\end{lemma}

\begin{lemma}\label{lm:CFMS}
	Let $D$ be a Lipschitz domain in $\RR^d$. For any $X\in D$, let $G(X, \cdot)$ denote the Green's function for the Laplacian in $D$ with pole at $X$, and let $\omega^X$ denote the corresponding harmonic measure. For any $p\in \partial D$ and $r>0$ such that $X \notin B_{2r}(p)$, we have that
	\[ \frac{ \omega^X(\Delta_r(p))}{r^{d-1}} \approx \frac{G(X, A_r(p))}{r} \]
	where $\approx$ means that the two quantities are equivalent modulo two universal constants.
\end{lemma}

\section{Lipschitz domains}\label{sec:Lip}
Let $H: \mathbb{R} \to \mathbb{R}$ be the Heaviside step function, namely
\[ H(x) = \left\{\begin{array}{ll}
	0, & x\leq 0 \\
	1, & x > 0
\end{array} \right. \]
Let $\{x_k\}$ be a sequence of distinct points in $\mathbb{R}\setminus\{0\} $ such that $x_k \to 0$. (Then in particular $\{x_k\}$ is bounded, say $|x_k| \leq 1$.) Let $c$ be a positive real number, and $\{a_k\}$ be a sequence in $\mathbb{R}_+$ such that
\begin{equation}\label{eq:akxk}
	c':= c \sum a_k < \frac{\pi}{2}. 
\end{equation} 
We define a function $f: \mathbb{R} \to \mathbb{R}$ as
\[ f(x) = c\sum a_k H(x-x_k). \]
Clearly
\[ \|f\|_{L^\infty} = c\sum a_k = c' < \frac{\pi}{2}. \]

Let $K$ be the Hilbert transform operator, modulo a constant, defined as
\begin{equation}\label{def:Kf}
	Kh(x) := \lim_{\epsilon \to 0} ~\frac{1}{\pi} \int h(y) \left[\frac{\chi_{|x-y|\geq\epsilon}}{x-y}  + \frac{\chi_{|y|>1}}{y}  \right] dy, 
\end{equation} 
whenever the limit on the right hand side exists. Here $\chi_E$ denotes the characteristic function of the set $E$. Recall that the Hilbert transform maps $L^\infty(\mathbb{R})$ functions into functions in the BMO space. Simple computations show that
\[ K\left(H(\cdot - x_k)\right)(x) = KH(x-x_k) = \frac{1}{\pi} \log|x-x_k|, \]
and hence, formally, we have
\[ Kf(x) = \frac{c}{\pi} \sum a_k \log|x-x_k|. \]
In fact, by the assumption \eqref{eq:akxk}, we have that
\[ c \sum_{k \leq \ell} a_k H(x-x_k) \to f(x) \text{ in } L^\infty(\RR), \quad \text{ as } \ell \to +\infty. \]
Hence
\begin{equation}\label{tmp:KBMOlimit}
	Kf(x) = \lim_{\ell \to +\infty} K\left(c \sum_{k\leq \ell} a_k H(x-x_k) \right) = \lim_{\ell \to +\infty} \frac{c}{\pi} \sum_{k \leq \ell} a_k \log|x-x_k|,  
\end{equation} 
where the limit is taken in the BMO space.

The function on the right hand side of \eqref{tmp:KBMOlimit} is well-defined and continuous in $\mathbb{R}\setminus\{0\}$. 
In fact, assume that $x\neq 0$ and $x\neq x_k$ for every $k$. 
Since $\log|\cdot|$ is a continuous function in $\RR \setminus \{0\}$, it is uniformly continuous on compact subset of $\RR \setminus\{0\}$. Let $E$ be a compact subset of $\RR \setminus\{0\}$ containing $x$, such that $E \cap \{x_k: k\in \NN\} = \emptyset$. We have that for any $y\in E$
\[ \log|y-x_k| \to \log |y|, \quad \text{ as } x_k \to 0, \]
and the convergence is uniform.
In particular, there exists $k_0\in \NN$ depending on $E$, such that for any $k \geq k_0$ and $y\in E$, we have that
\[ \left| \log |y-x_k| \right| \leq \left| \log|y| \right| + 1. \]
Hence for any $m\geq \ell \geq k_0$, we have
\begin{equation}\label{tmp:Linftysmall}
	\left| \frac{c}{\pi} \sum_{k=\ell}^{m} a_k \log|y-x_k| \right| \leq  
	\left( c\sum_{k = \ell}^{m} a_k \right) \cdot \left( \left| \log|y| \right| + 1 \right)<+\infty. 
\end{equation} 
Therefore
\[ \frac{c}{\pi} \sum a_k \log|x-x_k| = \lim_{\ell \to +\infty} \frac{c}{\pi} \sum_{k \leq \ell} a_k \log|x-x_k| \]
is well-defined and continuous in $E$. 
Therefore we have shown that $Kf(x)$ is well-defined and continuous on $\RR \setminus\left(\{x_k\} \cup \{0\} \right)$. Also, we have
\begin{equation}\label{tmp:loggrowth}
	|Kf(x)| \leq \frac{\pi}{2} \log\left(|x|+1\right), \quad \text{ whenever } |x|\geq 2. 
\end{equation}

Moreover near each $x_k$, we have
\begin{align}
	Kf(x) & = \frac{c}{\pi} a_k \log|x-x_k| + \lim_{k_0\to +\infty} \frac{c}{\pi} \sum_{k'\leq k_0 \atop{k' \neq k} } a_{k'} \log|x-x_{k'}| \nonumber \\
	& =:  \frac{c}{\pi} a_k \log|x-x_k| + e_k(x).\label{eq:Kfnearxk}
\end{align}  
Since $x_{k'} \neq x_k$ for every $k'\neq k$, and $x_{k'} \to 0 \neq x_k$, we have that
\begin{equation}\label{def:dk}
	\delta_k := \inf\{ |x_{k'}-x_k|: k'\in \NN \text{ and } k'\neq k \}>0. 
\end{equation} 
Then as long as $0<|x-x_k|< \delta_k/2$, 
\begin{align*}
	\frac{c}{\pi} \sum_{k' \neq k } a_{k'} \left| \log|x-x_{k'}| \right| & = \frac{c}{\pi} \sum_{k' \neq k \atop{|x-x_{k'}|\leq 1 } } a_{k'} \left| \log|x-x_{k'}|\right| + \frac{c}{\pi} \sum_{k' \neq k \atop{|x-x_{k'}|> 1 } } a_{k'} \left| \log|x-x_{k'}|\right| \\
	& = \frac{c}{\pi} \sum_{k' \neq k \atop{|x-x_{k'}|\leq 1 } } a_{k'}  \log\frac{1}{|x-x_{k'}|} + \frac{c}{\pi} \sum_{k' \neq k \atop{|x-x_{k'}|> 1 } } a_{k'}  \log|x-x_{k'}| \\
	& \leq \frac{c}{\pi} \sum_{k' \neq k \atop{|x-x_{k'}|\leq 1 } } a_{k'} \log \frac{2}{\delta_k} + \frac{c}{\pi} \sum_{k' \neq k \atop{|x-x_{k'}|> 1 } } a_{k'} \log \left( |x|+1 \right) \\
	& \leq \frac{c'}{\pi} \left( \log \frac{2}{\delta_k} + \log\left(|x|+1\right) \right).
\end{align*}
In particular, the second term in \eqref{eq:Kfnearxk}, $e_k(\cdot)$, is bounded near $x_k$, and thus
\begin{equation}\label{eq:decayKf}
	Kf(x) \to -\infty \text{ as } x\to x_k.
\end{equation} 

Let $V(x,t)$ and $W(x,t)$ be the Poisson integrals of $f(x)$ and $Kf(x)$, respectively, in the upper half plane $\mathbb{R}^2_+$. (The Poisson integral of $Kf(x)$ is well-defined, since by \eqref{tmp:loggrowth} $Kf(x)$ has logarithmic growth at infinity.) Since $f(x)$ is continuous and bounded in $\RR\setminus \left( \{x_k\}\cup\{0\} \right)$, by the classical theory\footnote{See \cite[Chapter 1 \S 2]{CBMS} for example.} for every $x\in \RR\setminus\left( \{x_k\}\cup\{0\} \right)$, we have that $V(z) \to f(x)$ as $z\to x$. Since $Kf(x)$ is continuous in $\RR \setminus\left( \{x_k\}\cup\{0\}\right)$, we also have that $W(z) \to Kf(x)$ as $z\to x$ for every $x\in \RR \setminus\left( \{x_k\}\cup\{0\}\right)$. 
Moreover, let \[ F_\ell(y) := \frac{c}{\pi} \sum_{k\leq \ell} a_k \log|y-x_k|, \]
and recall that $F_\ell \to Kf$ in $BMO(\RR)$ and pointwise in $\RR \setminus\left(\{x_k\} \cup\{0\} \right)$.
 We claim that for any $(x,t) \in \RR^2_+$,
\begin{equation}\label{cl:Poissonintegcv}
	W(x,t) = P_t \ast Kf(x) = \lim_{\ell\to +\infty} P_t \ast F_\ell(x), 
\end{equation} 
where $P_t$ denotes the Poisson kernel in $\RR^2_+$ with $P_t(\xi) = \frac{1}{\pi} \frac{t}{\xi^2+t^2}$. 
Then it immediately follows that
\begin{equation}\label{tmp:Wlog}
	W(x,t) = \sum P_t \ast \left(\frac{c}{\pi} a_k \log|x-x_k| \right) = \frac{c}{\pi}\sum  a_k \log|(x,t)-(x_k,0)|. 
\end{equation} 
In the second equality, we use the fact that the Poisson integral of $\log|x|$ in $\RR^2_+$ is just $\log|(x,t)|$.
To prove the claim, let
\[ \delta_0 : = \inf\{\left|(x,t)-(x_k,0) \right|: k\in \NN \}>0. \]
We let
\begin{align}
	P_t \ast h(x) & = P_t^1 \ast h(x) + P_t^2\ast h(x) \nonumber \\
	& \quad := \frac{1}{\pi} \int_{|y-x| \geq 2\delta_0 } \frac{t}{(x-y)^2+t^2} \, h(y) \, dy + \frac{1}{\pi} \int_{|y-x| < 2\delta_0 } \frac{t}{(x-y)^2+t^2} \, h(y) \, dy,\label{tmp:Ptcloseaway}
\end{align}
for any allowable function $h$ on $\RR$. For the first term, we have
\begin{align*}
	\left| P_t^1 \ast Kf(x) - P_t^1 \ast F_\ell(x) \right| & = \left| \frac{1}{\pi} \int_{|y-x| \geq 2\delta_0 } \frac{t}{(x-y)^2+t^2} \, \left(\frac{c}{\pi} \sum_{k>\ell} a_k \log|y-x_k| \right) \, dy \right| \\
	& \leq  \frac{1}{\pi} \int_{|y-x| \geq 2\delta_0 } \frac{t}{(x-y)^2+t^2} \,  \frac{c}{\pi} \sum_{k>\ell} a_k \left| \log|y-x_k| \right| \, dy \\
	& \leq  \frac{1}{\pi} \left(\frac{c}{\pi} \sum_{k>\ell} a_k \right) \int \frac{t}{(x-y)^2+t^2} \,  \max\left\{\log \frac{1}{\delta_0}, \log\left(|y|+1 \right) \right\} \, dy,
\end{align*}
which converges to $0$ as $\ell \to +\infty$. On the other hand, since $|x_k| \leq 1$ and $\log|\cdot|$ is square-integrable near the origin, for any $E$ which is a neighborhood of the origin, we have by the Minkowski integral inequality that
\[ \| Kf - F_\ell\|_{L^2(E)} =\left\| \frac{c}{\pi} \sum_{k>\ell} a_k \log|y-x_k| \right\|_{L^2(E)} \leq \frac{c}{\pi} \sum_{k>\ell} a_k \left\|\log|y-x_k| \right\|_{L^2(E)} \lesssim \frac{c}{\pi} \sum_{k>\ell} a_k \to 0, \]
as $\ell \to +\infty$. 
Hence
\begin{align*}
	\left|P_t^2 \ast Kf(x) - P^2_t \ast F_\ell(x) \right| & \leq \|Kf-F_\ell\|_{L^2(B_{|x|+2\delta_0}(0) )} \cdot \frac{1}{\pi} \left(  \int_{|y-x|\leq 2\delta_0} \left(\frac{t}{(x-y)^2+t^2} \right)^2 \, dy \right)^{1/2} \\
	& \leq \frac{1}{\pi} \left( \frac{2}{\delta_0} \right)^{1/2} \|Kf-F_\ell\|_{L^2(B_{|x|+2\delta_0}(0) )},
\end{align*}
which also converges to $0$ as $\ell \to +\infty$. This finishes the proof of the claim \eqref{cl:Poissonintegcv}.
In particular, \eqref{tmp:Wlog} implies that
\[ W(x,t) \to -\infty \quad \text{ as } (x,t) \text{ converges to } x_k, \]
with logarithmic decay.

Let $z=x+iy$, and we define the function
\[ g(z) := -W(x,y) + i V(x,y). \]
One can verify that $-W, V$ satisfy the Cauchy-Riemann equations and thus $g$ is analytic in $\RR^2_+$. 
Let $G(z):= \exp g(z)$. Then $G$ is analytic in $\RR^2_+$ and has non-tangential limit towards the boundary for almost every $x\in \partial \RR^2_+$, since $G$ is non-tangentially bounded at almost every boundary point. 
\[ \text{ For every } z\in \RR^2_+, \quad |G(z)| = \exp \left(- W(x,y) \right) \neq 0, \]
\begin{equation}\label{eq:Gx1}
	\text{ for every } x\in \RR \setminus \left( \{x_k\}\cup\{0\} \right), \quad |G(x)| = \exp\left(- Kf(x)\right) = \prod |x-x_k|^{-\frac{c}{\pi} a_k}, 
\end{equation} 
\begin{equation}\label{eq:Gx2}
	\text{ as } z \text{ converges to } x_k, \quad |G(z)| \approx \left|z-(x_k,0)\right|^{-\frac{c}{\pi} a_k} \to +\infty, 
\end{equation} 
and
\begin{equation}\label{eq:arg}
	\left|\arg G(z) \right| = \left|V(x,y) \right| \leq \|f\|_{L^\infty} \leq c'< \frac{\pi}{2}, 
\end{equation} 
where we used the maximum principle for the Poisson integral.

Let $\Phi$ denote the antiderivative of $G$ in $\RR^2_+$. More precisely, for any $z \in \RR^2_+$, let $\gamma_z$ denote any rectifiable curve from $i$ to $z$ and let
\[ \Phi(z) = \int_{\gamma_z} G. \]
This function is well-defined (i.e. independent of the choice of curve) since $\RR^2_+$ is simply connected. Besides, for any $z=(x,t) \in \RR^2_+$, by choosing $\gamma_z$ to be the line segment connecting $i$ to $z$, we can easily show that
\[ |\Phi(z)| = |z-i| \left|\int_0^1 G(\gamma_z(s)) \, ds \right| \leq |z-i| \cdot \min\{ t, 1 \}^{-\frac{c'}{\pi}} < +\infty, \]
namely $\Phi(z) \in \mathbb{C}$.
Since $|\Phi'(z)| = |G(z)| \neq 0 $, $\Phi$ is locally a conformal mapping. We claim that $\Phi$ is injective. Assume there are two distinct points $z_1, z_2\in \RR^2_+$ such that $\Phi(z_1) = \Phi(z_2)$. Let $\gamma_0$ denote the line segment in $\RR^2_+$ connecting $z_1$ to $z_2$. More precisely, we consider the parametrization $\gamma_0(t) = z_1 + t(z_2 - z_1)$ with $t\in [0,1]$. We have that
\[ \int_{\gamma_0} G = \Phi(z_2) - \Phi(z_1) = 0, \]
and that
\[ \int_{\gamma_0} G = (z_2-z_1) \cdot \int_0^1 G(\gamma_0(t)) \, dt. \]
Hence it follows that
\[ \int_0^1 G(\gamma_0(t)) \, dt = 0. \]
In particular, the real part of the above integral also vanishes, i.e.
\begin{equation}\label{eq:realinj}
	\int_{0}^1 |G(\gamma_0(t))|\cos\left(\arg G(\gamma_0(t)) \right) \, dt = 0. 
\end{equation} 
However, by \eqref{eq:arg} we have that
\[ \cos\left( \arg G(z) \right) \geq \cos c' >0\quad \text{ for all } z\in \RR^2_+. \]
Combined with $G(z) \neq 0$, this is a contradiction with \eqref{eq:realinj}. Therefore $\Phi$ is injective.

For any $z\in \partial \RR^2_+ $ and $z\neq x_k, z\neq 0$, by the properties of the Poisson integrals $V$ and $W$ it is easy to see that 
\[ \Phi(z) = \int_{\gamma_z} G \text{ is still well-defined} \]
and moreover, it is independent of the choice of the curve $\gamma_z \subset \overline{\RR^2_+}$.
Next, we show that $\Phi(z)$ is well-defined as $z\to x_k$ and as $z\to 0$, and is independent of the choice of the curve. 
For fixed $k$, let $z, z'$ be arbitrary points in $\overline{\RR^2_+} \setminus \{0\}$ with $z, z' \neq x_{k'}$ for any $k'\in \mathbb{N}$ 
 and such that $\delta:= \max\left\{|z-(x_k,0)|, |z'-(x_k,0)| \right\}$ is sufficiently small. Let $\gamma$ denote a rectifiable curve in $\overline{\RR^2_+}$ connecting $z$ and $z'$, such that $\gamma$ does not intersect the origin or $x_{k'}$ for any $k' \in \mathbb{N}$. Then by \eqref{eq:Gx2},
\[ \left| \Phi(z) - \Phi(z') \right| = \left| \int_{\gamma} G \right| 
 \lesssim \int_{\gamma} |\gamma(t) - (x_k,0)|^{-\frac{c}{\pi} a_k}.  \]
Since $\frac{c}{\pi} a_k < \frac12$ by the assumption \eqref{eq:akxk}, by carefully choosing the curve $\gamma$ (for example, by taking $\gamma$ to be the union of an arc on $\partial B_{\delta}(x_k)$ and a line segment on a ray from $x_k$) we can guarantee that 
\[ \int_{\gamma} |\gamma(t) - (x_k,0)|^{-\frac{c}{\pi} a_k} \to 0 \quad \text{ as } \delta \to 0. \]
Therefore $\Phi(z)$ is continuous and finite as $z\to x_k$. 
To show $\Phi(z)$ is continuous at the origin, let $z = (x_0,t_0), z'=(x'_0,t'_0)$ be arbitrary points in $\overline{\RR^2_+} \setminus\{0\}$ that are sufficiently close to the origin. Let $\delta:= \max\{|z|, |z'|\} $, and clearly $|x_0|, |x'_0|, t_0, t'_0 < \delta$. Let $t_* := \max\{t_0+\delta, t'_0+\delta \}$. Let $\gamma_1$ denote the vertical line segment between $z$ and $(x_0, t_*)$, $\gamma_2$ denote the horizontal line segment between $(x_0, t_*)$ and $(x'_0, t_*)$, and $\gamma_3$ denote the vertical line segment between $(x'_0, t_*)$ and $z'$, each parametrized by unit length. For each $i\in \{1,2,3\}$, we have
\begin{align*}
	\left| \int_{\gamma_i} G \right| \leq \int_{\gamma_i} |G| = \int_{\gamma_i} \exp\left(-W(z) \right) & = \int_{\gamma_i} \exp\left(-\sum \frac{c}{\pi} a_k \log |z-(x_k,0)| \right) \\
	& = \int_{\gamma_i} \prod |z-(x_k,0)|^{-\frac{c}{\pi} a_k}.
\end{align*} 
Since we always have that
\[ |z-(x_k,0)| \geq \Imag z, \]
it follows that
\begin{align*}
	\int_{\gamma_1} \prod |z-(x_k,0)|^{-\frac{c}{\pi} a_k} \leq \int_{t_0}^{t_*} t^{-\sum \frac{c}{\pi} a_k} \lesssim t_*^{1-\frac{c}{\pi} \sum a_k} \lesssim \delta^{1-\frac{c}{\pi} \sum a_k},
\end{align*}
\begin{align*}
	\int_{\gamma_2} \prod |z-(x_k,0)|^{-\frac{c}{\pi} a_k} \leq \int_{\gamma_2} t_*^{-\sum \frac{c}{\pi} a_k} = |x'_0-x_0| t_*^{-\sum \frac{c}{\pi} a_k} \lesssim \delta^{1-\frac{c}{\pi} \sum a_k},
\end{align*}
and the same estimate holds for $\gamma_3$. Therefore
\begin{align*}
	\left| \Phi(z) - \Phi(z') \right| = \left|\int_{\gamma_1 \cup \gamma_2 \cup \gamma_3} G \right| \leq \sum_{i=1}^3 \int_{\gamma_i} |G| \lesssim \delta^{1-\frac{c}{\pi} \sum a_k}.
\end{align*}
Hence $\Phi(z)$ is continuous and finite as $z\to 0$. To sum up, we have shown that $\Phi$ has a continuous extension to $\overline{\RR^2_+}$. Using the same argument \eqref{eq:realinj} as in the interior case, we can show that $\Phi$ is also injective on $\overline{\RR^2_+}$.


We claim that $\Phi(\infty) = \infty$. Thus in particular, the set $D:=\Phi(\RR^2_+)$ is unbounded, and $\partial D = \Phi(\partial \RR^2_+)$ (the boundary in $\RR^2_+$, not in the Riemann sphere). 
Let $z_j$ be an arbitrary sequence in $\RR^2_+$ such that $z_j \to \infty$. Let $\gamma_j$ denote the straight line segment connecting $i$ to $z_j$, namely $\gamma_j(t) = i+t(z_j-i)$ for $t\in [0,1]$. 
Then
\[ \Phi(z_j) = \int_{\gamma_j} G = (z_j-i) \cdot \int_0^1 G(\gamma_j(t))\, dt. \]
Hence
\[ \Real \frac{\Phi(z_j)}{z_j - i} = \Real \int_0^1 G(\gamma_j(t))\, dt = \int_0^1 |G(\gamma_j(t))| \cos\left(\arg G(\gamma_j(t)) \right) dt  \geq \cos c' \cdot \int_0^1 |G(\gamma_j(t))| \, dt. \]
Using \eqref{tmp:Wlog} again, we have
\[ \int_0^1 |G(\gamma_j(t))| \, dt = \int_0^1 \prod |\gamma_j(t) - (x_k,0)|^{-\frac{c}{\pi} a_k} \, dt \geq \int_0^1 \left(|\gamma_j(t)| + 1 \right)^{-\frac{c'}{\pi}} \, dt \geq \left(|z_j|+1 \right)^{-\frac{c'}{\pi}}. \]
Therefore
\begin{equation}\label{tmp:provePhiinfty}
	|\Phi(z_j)| = \left| z_j - i \right| \cdot \left| \frac{\Phi(z_j)}{z_j - i} \right| \geq \left| z_j - i \right| \cdot \left| \Real \frac{\Phi(z_j)}{z_j - i} \right| \gtrsim |z_j|^{1-\frac{c'}{\pi}},
\end{equation}
for $j$ sufficiently large. In particular $\Phi(z_j) \to \infty$ for any sequence $z_j$ in $\RR^2_+$ such that $z_j \to \infty$.

Since $\Phi: \RR^2_+ \to D$ is a conformal homeomorphism, and $\Phi$ is injective on $\partial \RR^2_+$ with $\partial D = \Phi(\partial \RR^2_+)$, it follows that $D$ is also simply connected and bounded by a simple curve. By the same argument as in \cite[Theorem 1.1]{Kenig}, 
at every $x\in \partial \RR^2_+$ where $\Phi'(x)$ exists and is different from $0$, it is a tangent vector to $\partial D$ at the point $p:=\Phi(x)$;
and a set $E\subset \partial \RR^2_+ $ has measure zero if and only if $\Phi(E)\subset \partial D$ has surface measure zero.
We remark that since $G(x)$ is continuous in $\RR \setminus \left( \{x_k\}\cup\{0\} \right)$, by the fundamental theorem of calculus $\Phi'(x) = G(x)$ there. 
Moreover, 
Let $[x_k-a, x_k+b]\subset \RR $ be an arbitrary interval containing $x_k$, with $a, b$ sufficiently small satisfying $0<a,b<\delta_k/2$ (recall the definition of $\delta_k$ in \eqref{def:dk}).
Recall that $\prod\limits_{k' \neq k} |x-x_{k'}|^{-\frac{c}{\pi} a_{k'}}$ is continuous at $x_k$. It follows that by choosing $a, b$ sufficiently small, we can guarantee that
\[ \prod\limits_{k' \neq k} |x-x_{k'}|^{-\frac{c}{\pi} a_{k'}} > \frac12 \prod\limits_{k' \neq k} |x_k-x_{k'}|^{-\frac{c}{\pi} a_{k'}} > 0 \quad \text{ for every } x\in  [x_k-a, x_k+b]. \]
Hence
\begin{align*}
	\fint_{x_k-a}^{x_k+b} |\Phi'(x)| ~dx = \fint_{x_k-a}^{x_k+b} |G(x)| ~dx & = \fint_{x_k-a}^{x_k+b} \prod |x-x_k|^{-\frac{c}{\pi} a_k} ~dx \\
	& \geq \fint_{x_k-a}^{x_k+b} |x-x_k|^{-\frac{c}{\pi} a_k} ~dx \cdot \frac12 \prod\limits_{k' \neq k} |x_k-x_{k'}|^{-\frac{c}{\pi} a_{k'}} \\
	& \gtrsim \frac12 \prod\limits_{k' \neq k} |x_k-x_{k'}|^{-\frac{c}{\pi} a_{k'}} \cdot \max\{a,b\}^{-\frac{c}{\pi} a_k}.
\end{align*}
In particular,
\begin{equation}\label{eq:avgPhider}
	\fint_{x_k-a}^{x_k+b} |\Phi'(x)| ~dx \to +\infty, \quad \text{ as } a, b \to 0+.
\end{equation}

Let $\omega$ denote the harmonic measure in $D$ with pole at infinity and normalized at $\Phi(0)$ (for its precise definition and properties, see \cite[Corollary 3.2 and Lemma 3.8]{KT}).
By the conformal invariance of the Brownian motion, we can find the explicit formula for $\omega$: 
we have that
\[ d\omega(z) = \frac{1}{|\Phi'(\Phi^{-1}(z))|} d\sigma(z), \]
where $\sigma$ denotes the surface measure at the boundary $\partial D$, namely $\sigma = \mathcal{H}^1 |_{\partial D}$.
In fact, let $\omega_{\RR^2_+}$ denote the harmonic measure of $\RR^2_+$ with pole at infinity and normalized at the origin. Clearly $d\omega_{\RR^2_+} = dx$. For any $z\in \partial D$, let $\Delta$ denote a surface ball of $D$ centered at $z$. 
 Then for every $z\in \partial D$ such that $\Phi^{-1}(z) \notin \{x_k\} \cup\{0\}$, as $\Delta \to z$, we have
\begin{equation}\label{eq:omegaPhider}
	\frac{\omega(\Delta)}{ \mathcal{H}^1(\Delta)} = \frac{\omega_{\RR^2_+}(\Phi^{-1}(\Delta)) }{ \int_{\Phi^{-1}(\Delta)} |\Phi'(x)| dx } = \frac{\int_{\Phi^{-1}(\Delta)} dx }{ \int_{\Phi^{-1}(\Delta)} |\Phi'(x)| dx } \longrightarrow \frac{1}{|\Phi'(\Phi^{-1}(z))|} 
	= \frac{1}{|G(\Phi^{-1}(z))|}>0. 
\end{equation} 
Here we use the conformal invariance of the harmonic measure and the area formula in the first equality.
On the other hand, when $z=\Phi(x_k)$ for some $x_k$, we claim that 
\[ 
\lim_{\Delta \to z} \frac{\omega(\Delta)}{\mathcal{H}^1(\Delta)} = 0. \] In fact, since $\Phi$ is a homeomorphism on $\partial \RR^2_+$, $\Phi^{-1}(\Delta)$ is just an interval containing $\Phi^{-1}(z)= x_k$. By \eqref{eq:avgPhider}, it follows that
\[ \lim_{\Delta \to z} \frac{\omega(\Delta)}{\mathcal{H}^1(\Delta)} = \lim_{a, b \to 0+ } \left( \fint_{x_k-a}^{x_k+b} |\Phi'(x)| ~dx \right)^{-1} = 0. \]
In short,
\[ \left\{z\in \partial D: \lim_{\Delta \to z} \frac{\omega(\Delta)}{\mathcal{H}^1(\Delta)} 
= 0\right\} \setminus\{\Phi(0)\} = 
~ \Phi\left(\{z_k: k\in \NN\} \right). \]

Lastly, we remark that given the input $f(x):= cH(x)$, where $c$ is a constant with $0<c<\pi/2$ and $H(x)$ is the Heaviside step function, our construction produces the following simple Lipschitz domain. Intuitively it is clear that the density of the harmonic measure in $D$ is zero only at the vertex.
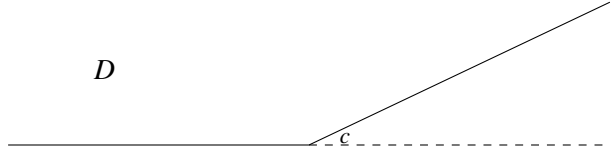
\begin{figure}[h!]
    \begin{tikzpicture}
    	\draw (-4,0) -- (0, 0) -- (4, 1.9);
    	\draw [dashed] (0, 0) -- (4, 0);
    	\node[right] at (0.3, 0.1) [scale=.8] {$c$};
    	\node[right] at (-3,1) {$D$};
    \end{tikzpicture}
\end{figure}

\section{$C^1$ domains}\label{sec:C1}

The following lemma is just a special case of the more general Lemma \ref{lm:generalDinifunction} that we need later. But we introduce and prove this lemma first, in order to fix ideas.
\begin{lemma}\label{lm:Dinifunction}
	There exists a continuous function $f\in C(\RR)$ such that
	\begin{enumerate}
		\item\label{it:1} $f$ is a monotone non-decreasing function with $0\leq f\leq 1$, and $f \in C^1(\RR\setminus\{0\})$;
		\item\label{it:2} the modulus of continuity of $f$ at the origin, denoted by $\theta(r)$, satisfies 
			\[ \int_0^* \frac{\theta(r)}{r} dr = +\infty; \]
		\item\label{it:3} $Kf(x) \in C(\RR \setminus \{0\} )$, and $Kf(x) \to Kf(0) = -\infty$ as $x\to 0$, where $K$ denotes the Hilbert transform operator as is defined in \eqref{def:Kf}.
	\end{enumerate}
\end{lemma}
\begin{proof}
	Let $\theta: (0,1) \to \RR_+$ be defined as $\theta(r) = \left(\log_2 \frac{1}{r} \right)^{-1}$. Clearly $\theta$ is monotone increasing, $\lim_{r\to 0+} \theta(r) = 0$, and
	\[ \int_0^* \frac{\theta(r)}{r} dr = +\infty. \] 
	One can check that there exists $x_0 = 2^{-\frac{1}{\log 2}} \in (0, 1/2)$ such that in $(0, x_0]$, $\theta$ is concave, and 
	\begin{equation}
		\theta'(x) = \frac{1}{x\cdot \log 2} \left( \log_2 \frac{1}{x} \right)^{-2} > 0 \text{ is monotone decreasing}.\label{eq:derf}
	\end{equation} 
	Let $g(\cdot)$ be a smooth, non-decreasing function defined on $[x_0, 1/2]$ such that
\[ g\left(x_0 \right) = \theta \left(x_0 \right) =  \log 2, \quad g\left( \frac12 \right) = \theta\left( \frac12 \right) = 1,  \]
\[ g'\left(x_0+ \right) = \theta'\left( x_0 \right) = 2^{\frac{1}{\log 2}}  \log 2, \quad g'\left( \frac12- \right) = 0, \]
and $|g'(x)| \leq \|g'\|_\infty = 2^{\frac{1}{\log 2}}$ for all $x\in [x_0, 1/2]$.
Finally we define $f: \RR \to \RR$ as follows
\[ f(x) = \left\{\begin{array}{ll}
	0, & x\leq 0 \\
	\theta(x) = \left(\log_2 \frac{1}{|x|} \right)^{-1}, & 0<x  \leq x_0 \\
	g(x) , & x_0 < x < \frac12 \\
	1, & x \geq \frac12
\end{array} \right. \]
It is not hard to see that $f$ satisfies \eqref{it:1} and \eqref{it:2}. Next, we analyze $Kf(x)$. Recall that the Hilbert transform maps bounded continuous functions into functions in the VMO space, so $Kf(x) \in VMO(\RR)$.

When $x<0$, we have that
\begin{equation}\label{eq:Kfx<0}
	\pi \cdot Kf(x) = \int_0^{1/2} \frac{f(y)}{x-y} \, dy + \log \left( \frac12-x \right), 
\end{equation} 
and clearly $Kf(x)$ is continuous on $(-\infty, 0)$.
A rough estimate (simply using the monotonicity of $f(\cdot)$) gives 
\begin{equation}\label{eq:x<0}
-\infty < (1-f(x_0)) \log(x_0-x) + f(x_0) \log (-x) \leq \pi\cdot Kf(x) \leq \log \left( \frac12 - x \right)< +\infty.
\end{equation} 

Moreover, we claim that
\[ \int_0^{1/2} \frac{f(y)}{x-y} \, dy \to -\infty \quad \text{ as } x\to 0-. \]
Combined with \eqref{eq:Kfx<0}, the claim implies that
\[ Kf(x) \to -\infty \quad \text{ as } x\to 0-. \]
In fact, since
\begin{equation}\label{tmp:Kfx<0}
	\int_0^{1/2} \frac{f(y)}{x-y} \, dy = -\int_0^{x_0} \frac{\theta(y)}{y-x} \, dy - \int_{x_0}^{1/2} \frac{g(y)}{y-x} \, dy 
\end{equation} 
and the second term is uniformly bounded in $x$, it suffices to show that
\[ \int_0^{x_0} \frac{\theta(y)}{y-x} \, dy \to +\infty \quad \text{ as } x\to 0-. \]
This follows easily from Fatou's Lemma:
\[ +\infty = \int_0^{x_0} \frac{\theta(y)}{y} \, dy \leq \liminf_{x\to 0-} \int_0^{x_0} \frac{\theta(y)}{y-x} \, dy \]


When $0<x<x_0$, we have that
\begin{align}
	\pi\cdot Kf(x) = \lim_{\epsilon\to 0} \left(\int_0^{x-\epsilon}  \frac{f(y)}{x-y} dy + \int_{x+\epsilon}^{x_0} \frac{f(y)}{x-y} dy \right) +\int_{x_0}^{\frac12}  \frac{f(y)}{x-y} dy + \log\left(\frac12 - x \right).\label{eq:Htmidt1}
\end{align}
Notice that
\begin{align}
	\lim_{\epsilon\to 0} \left( \int_0^{x-\epsilon}  \frac{f(y)}{x-y} dy + \int_{x+\epsilon}^{x_0} \frac{f(y)}{x-y} dy \right) = f(x) \cdot \left[ \log x - \log\left( x_0 - x \right) \right] + \int_0^{x_0 } \frac{f(y) - f(x)}{x-y} dy, \label{eq:Htmidt2}
\end{align}
when the integral on the right hand side is well-defined.
In order to analyze the last term $\int_0^{x_0 } \frac{f(y) - f(x)}{x-y} dy$, we break the integral into two regions $y\in [0,x]$ and $y\in [x, x_0]$. On one hand, by the mean value theorem and the monotonicity of $f'(\cdot)$ on $[0,x_0]$ we have
\begin{align}
	0< \int_x^{x_0 } \frac{f(y) - f(x)}{y-x} dy 
	\leq \sup_{[x,x_0] } f' \cdot (x_0-x) = f'(x)\cdot (x_0-x).\label{eq:Htmid1}
\end{align}
On the other hand
\begin{align}
	0 < \int_0^{x } \frac{f(y) - f(x)}{y-x} dy & = \int_0^{x/2 } \frac{f(y) - f(x)}{y-x} dy + \int_{x/2}^{x } \frac{f(y) - f(x)}{y-x} dy. \label{eq:Htmid2}
\end{align}
Since $0\leq f\leq f(x_0)$ on $[0,x]$, we can control the first term:
\begin{equation}\label{eq:Htmid3}
	0< \int_0^{x/2 } \frac{f(y) - f(x)}{y-x} dy \leq \int_0^{x/2 } \frac{f(x)}{x-y} dy \leq f(x_0) \cdot \int_0^{x/2} \frac{1}{x-y} \, dy = f(x_0) \cdot \log 2;
\end{equation}
again by the mean value theorem and the monotonicity of $f'(\cdot)$ on $[0,x_0]$, we can control the second term:
\begin{equation}
	0 < \int_{x/2}^{x } \frac{f(y) - f(x)}{y-x} dy \leq \int_{x/2}^{x} \sup_{[x/2,x]} f' \, dy \leq f'\left( \frac{x}{2} \right) \cdot \frac{x}{2}.\label{eq:Htmid4}
\end{equation} 
Combining \eqref{eq:Htmid1}, \eqref{eq:Htmid2}, \eqref{eq:Htmid3} and \eqref{eq:Htmid4}, we conclude that
\begin{equation}
0 < \int_0^{x_0 } \frac{f(y) - f(x)}{y-x} dy \leq f'(x) \cdot (x_0 - x) + f'\left(\frac{x}{2}\right) \cdot \frac{x}{2} + f(x_0) \cdot \log 2. \label{eq:Htmidt3}
\end{equation} 
Finally, combining \eqref{eq:Htmidt1}, \eqref{eq:Htmidt2}, \eqref{eq:Htmidt3}, 
we obtain
\begin{align}
	\pi\cdot Kf(x) & \geq
	(1-f(x)) \log (x_0-x) + f(x)\log x - f'(x) \cdot (x_0-x) -f'\left( \frac{x}{2} \right) \cdot \frac{x}{2} - f(x_0) \cdot \log 2 \nonumber \\
	& > -\infty,\label{eq:x>0lb}
\end{align}
and
\begin{align}
	\pi \cdot Kf(x) & \leq (1-f(x_0)) \log \left(\frac12 - x \right) + (f(x_0)-f(x)) \log (x_0-x) + f(x) \log x \nonumber \\
	& < +\infty.\label{eq:x>0ub}
\end{align}
%
%

We claim that
\[ Kf(x) \to -\infty \quad \text{ as } x\to 0+. \]
By \eqref{eq:Htmidt1}, \eqref{eq:Htmidt2} and the dominated convergence theorem (which implies that $\lim_{x\to 0+} \int_{x_0}^{1/2} \frac{f(y)}{x-y} dy $ $ = -\int_{x_0}^{1/2} \frac{f(y)}{y} dy $ and is finite), to prove the claim it suffices to show 
\[ f(x) \log x + \int_0^{x_0} \frac{f(y)-f(x)}{x-y} \, dy = 
 - f(x) \log \frac{1}{x} - \int_0^{x_0} \frac{f(y)-f(x)}{y-x} \,dy 
\to -\infty, \]
as $x\to 0+$.
This holds because
\[ f(x) \log \frac{1}{x} > 0 
\]
and
\[ \liminf_{x\to 0+} \int_0^{x_0} \frac{f(y)-f(x)}{y-x} \,dy \geq \int_0^{x_0} \frac{\theta(y)}{y} \, dy = +\infty, \]
by Fatou's Lemma (since $\frac{f(y)-f(x)}{y-x} \in [0,+\infty]$ for every $y\in [0,x_0]$) and the fact that $\lim_{x\to 0} f(x) = f(0) = 0$.
Lastly, we also remark that since we have proven the right hand side of \eqref{eq:Htmidt2}, as a principal value, is finite, we can formally write
\begin{equation}\label{eq:Htmidt4}
	\int_0^{x_0} \frac{f(y)}{x-y} \, dy : = f(x) \cdot \left[ \log x - \log\left( x_0 - x \right) \right] + \int_0^{x_0 } \frac{f(y) - f(x)}{x-y} dy, 
\end{equation} 
which is well-defined for every $0<x<x_0$ and decays to $-\infty$ as $x\to 0+$. (Recall that by \eqref{eq:Kfx<0} and \eqref{tmp:Kfx<0}, the decay rate of $Kf(x)$ as $x\to 0-$ is also given by $\int_0^{x_0} \frac{f(y)}{x-y} dy$.)

When $x=x_0$, we have
\begin{align}
	\pi\cdot Kf(x_0) & =\lim_{\epsilon \to 0} \left( \int_0^{x_0-\epsilon} \frac{f(y)}{x_0-y} \, dy + \int_{x_0+\epsilon}^{1/2} \frac{f(y)}{x_0-y} \, dy \right) + \log\left( \frac12 - x_0 \right) \nonumber \\
	& = \lim_{\epsilon \to 0} \left( \int_0^{x_0-\epsilon} \frac{f(x_0)}{x_0-y} \, dy + \int_{x_0+\epsilon}^{1/2} \frac{f(x_0)}{x_0-y} \, dy \right) + \int_0^{x_0} \frac{f(y)-f(x_0)}{x_0-y} \, dy \nonumber \\
	& \qquad + \int_{x_0}^{1/2} \frac{f(y)-f(x_0)}{x_0-y} \, dy + \log\left( \frac12 - x_0 \right) \nonumber \\
	& = f(x_0)  \log x_0 + (1-f(x_0)) \log \left( \frac12 - x_0 \right) \nonumber \\
	& \qquad  + \int_0^{x_0} \frac{f(y)-f(x_0)}{x_0-y} \, dy + \int_{x_0}^{1/2} \frac{f(y)-f(x_0)}{x_0-y} \, dy.\label{eq:Kfx0}
\end{align}
For the last two terms, notice that
\[ 0< \int_{x_0}^{1/2} \frac{f(y)-f(x_0)}{y-x_0} \, dy \leq \|g'\|_\infty \left( \frac12 - x_0 \right), \]
\[ 0<\int_0^{x_0/2} \frac{f(x_0)-f(y)}{x_0-y} \, dy \leq \int_0^{x_0/2} \frac{f(x_0)}{x_0-y} \, dy = f(x_0) \cdot \log 2,  \]
and by the concavity of $f$ on $(0,x_0)$,
\begin{equation}\label{eq:estx=x0}
	0 \leq  \int_{x_0/2}^{x_0} \frac{f(x_0)-f(y)}{x_0-y} \, dy \leq f'\left( \frac{x_0}{2} \right) \cdot \frac{x_0}{2}. 
\end{equation} 
Therefore
\begin{align*}
	\pi\cdot Kf(x_0) & \geq f(x_0) \log \frac{x_0}{2} + (1-f(x_0)) \log \left( \frac12 - x_0 \right) - \|g'\|_\infty \left( \frac12 - x_0 \right) - f'\left( \frac{x_0}{2} \right) \cdot \frac{x_0}{2} \\
	& > -\infty,
\end{align*} 
\[ \pi \cdot Kf(x_0) \leq f(x_0)  \log x_0 + (1-f(x_0)) \log \left( \frac12 - x_0 \right) 
< +\infty. \]
Moreover, by \eqref{eq:Htmidt1}, \eqref{eq:Htmidt2} and \eqref{eq:Kfx0}, we also have
\begin{equation}\label{eq:Kfx0l}
	Kf(x) \to Kf(x_0) \quad \text{ as } x \to x_0-. 
\end{equation}

When $x_0<x<1/2$, we have
\begin{align}
	\pi\cdot Kf(x) & = \int_0^{x_0} \frac{f(y)}{x-y} dy + \lim_{\epsilon \to 0}\left(  \int_{x_0}^{x-\epsilon} \frac{f(y)}{x-y} dy + \int_{x+\epsilon}^{1/2} \frac{f(y)}{x-y} dy \right) + \log\left( \frac12 - x \right) \nonumber \\
	& = \int_0^{x_0} \frac{f(y)}{x-y} dy + f(x) \left[ \log(x-x_0) - \log \left(\frac12 - x \right) \right] \nonumber \\
	& \qquad \qquad + \int_{x_0}^{1/2} \frac{f(y)-f(x)}{x-y} dy + \log\left(\frac12 - x \right).\label{eq:Kfx>x0}
\end{align}
Notice that
\[ -\|g'\|_\infty \cdot \left( \frac12 - x_0 \right) \leq \int_{x_0}^{1/2} \frac{f(y)-f(x)}{x-y} dy \leq 0,  \]
and 
\[ 0 \leq \int_0^{x_0} \frac{f(y)}{x-y} dy \leq f(x_0) \cdot \int_0^{x_0} \frac{1}{x-y} dy = f(x_0) \left[ \log x - \log (x-x_0) \right]. \]
Therefore we have that
\begin{equation}\label{eq:Kfx>x0lb}
	\pi\cdot Kf(x) \geq f(x) \log(x-x_0) + (1-f(x)) \log \left( \frac12 - x \right) -\|g'\|_\infty\cdot \left( \frac12 - x_0 \right)> -\infty, 
\end{equation} 
and
\begin{equation}\label{eq:Kfx>x0ub}
	\pi\cdot Kf(x) \leq \left(f(x)-f(x_0) \right) \log(x-x_0) + f(x_0) \log x + (1-f(x)) \log\left(\frac12 - x \right)<+\infty. 
\end{equation} 
Moreover, by \eqref{eq:Kfx>x0} and \eqref{eq:Kfx0}, we have that
\begin{equation}\label{eq:Kfx0r}
	Kf(x) \to Kf(x_0) \quad \text{ as } x\to x_0+. 
\end{equation} 

When $x= 1/2$, we have
\begin{align}
	\pi\cdot Kf\left( \frac12 \right) & = \lim_{\epsilon \to 0} \int_0^{1/2-\epsilon} \frac{f(y)}{1/2 - y} \, dy + \int_{1/2}^\infty \left[ \frac{1}{1/2 - y}\chi_{\{y> 1/2 +\epsilon\}} + \frac{1}{y} \right] \, dy \nonumber \\
	& = \int_0^{x_0} \frac{f(y)}{1/2-y} \, dy + \int_{x_0}^{1/2} \frac{f(y)-f(1/2)}{1/2-y} \, dy + \log \left( \frac12 - x_0 \right).\label{eq:Kf1/2}
\end{align}
Therefore
\begin{equation}\label{eq:estx=1/2}
	-\infty< \log x_0 \leq \pi \cdot Kf\left( \frac12 \right) \leq -(1-f(x_0))\cdot \log 2 + f(x_0) \cdot \log x_0< +\infty.  
\end{equation} 
Moreover, by \eqref{eq:Kfx>x0}, \eqref{eq:Kf1/2} and the assumption $\lim_{x\to 1/2} f(x) = f(1/2) = 1$, we have that
\[ Kf(x) \to Kf\left( \frac12 \right) \quad \text{ as } x\to \frac12-. \]

When $x>1/2$,
\begin{equation}\label{eq:Kfx>1/2}
	\pi\cdot Kf(x) = \int_0^{\frac12}  \frac{f(y)}{x-y} dy + \log \left(  x-\frac12 \right).
\end{equation}
Hence we have
\begin{equation}\label{eq:estx>1/2}
	-\infty<   \log\left(x-\frac12 \right) \leq  \pi\cdot Kf(x) \leq  \log x - \log \left( x- \frac12 \right) + \log \left(x-\frac12 \right)  =  \log x < +\infty.
\end{equation} 
Moreover, since by Lemma \ref{lm:prelim},
\[ \log\left( x- \frac12 \right) - \log\left( \frac12 - x_0 \right) = f\left( \frac12 \right) \cdot \int_x^{x_0} \frac{1}{1/2-y} \, dy, \]
by combining \eqref{eq:Kf1/2} and \eqref{eq:Kfx>1/2} we show that
\[ Kf(x) \to Kf\left( \frac12 \right) \quad \text{ as } x\to \frac12+. \]
\end{proof}

Next, for any non-decreasing function $\theta$ satisfying \eqref{cond:theta}, we construct a continuous function whose modulus of continuity is given by $\theta$.
\begin{lemma}\label{lm:generalDinifunction}
	Let $\theta: \RR_+ \to \RR_+$ be a monotone non-decreasing function such that 
	\[ \lim_{r\to 0+} \theta(r) = 0 \quad \text{ and } \quad \int_0^* \frac{\theta(r)}{r} \, dt = +\infty. 
	\]
	Let $x_0>0$ be sufficiently small (depending on $\theta$).
	There exists $f\in C(\RR)$, defined as in \eqref{def:Htilde}, which satisfies all the properties in Lemma \ref{lm:Dinifunction}, and moreover, the modulus of continuity of $f$ at the origin, denoted by $\tilde{\theta}(r)$, satisfies
	\begin{equation}\label{eq:modcont}
		\theta(r) \leq \tilde{\theta}(r) \leq \theta(4r). 
	\end{equation} 
\end{lemma}
\begin{proof}
	Let
	\begin{equation}\label{def:thetatilde}
	\tilde{\theta}(r) = \frac{1}{\log^2 2} \int_r^{2r} \frac1t \int_t^{2t} \frac{\theta(s)}{s} ds ~dt.
	\end{equation}
	Simple computations show that
	\begin{equation}\label{eq:thetatilde}
		\theta(r) \leq \tilde{\theta}(r) \leq \theta(4r), \quad \lim_{r\to 0+} \tilde{\theta}(r) = 0, 
	\end{equation} 
	and
	\begin{align*}
		\frac{d}{dr} \tilde{\theta}(r) & = \frac{1}{\log^2 2} \cdot \frac{1}{r} \left[ \int_{2r}^{4r} \frac{\theta(s)}{s} ds -  \int_{r}^{2r} \frac{\theta(s)}{s} ds \right] \\
		& = \frac{1}{\log^2 2} \int_1^2 \frac{\theta(2rt) - \theta(rt)}{rt}\, dt \geq 0.
	\end{align*} 
	Let $x_*$ be the largest real number such that $\tilde{\theta}(r)< 1$ for all $r\in [0,x_*)$. (If $\tilde{\theta}(r)<1$ for all $r\in \RR_+$, we simply let $x_* = 1/2$.) Then for any $r\leq x_*/4$, we have
	\begin{equation}\label{eq:bdthetader}
		\frac{d}{dr} \tilde{\theta}(r) \leq \frac{1}{\log^2 2} \int_1^2 \frac{1}{rt} \, dt = \frac{1}{r\log 2}.
	\end{equation} 
	Let $x_0 \in (0, x_*/4)$ be sufficiently small such that $\tilde{\theta}(x_0)<1/2$ (other than this constraint, we are free to choose $x_0$ as small as needed). Let $g$ be a smooth, non-decreasing function defined on $[x_0, x_*]$, such that
	\[ g(x_0) = \tilde{\theta}(x_0), \quad g(x_*)  = 1, \]
	\[ g'(x_0+) = \frac{d}{dr} \tilde{\theta}(x_0), \quad g'(x_*-) = 0, \]
	and $|g'(r)| \leq \|g'\|_\infty$. We define the function $f: \RR\to \RR$ as follows:
	\begin{equation}\label{def:Htilde}
		f(x) = \left\{\begin{array}{ll}
		0, & x\leq 0 \\
		\tilde{\theta}(x), & 0<x \leq x_0 \\
		g(x), & x_0<x< x_* \\
		1, & x\geq x_*
	\end{array} \right. 
	\end{equation} 
	Clearly $f(x)$ satisfies (1) (2) of Lemma \ref{lm:Dinifunction}, \eqref{eq:modcont}, and $Kf(x) \in VMO(\RR)$ since $f$ is a bounded continuous function on $\RR$.
	
	In the proof of the property (3) in Lemma \ref{lm:Dinifunction}, we use the fact that on $[0,x_0]$, the function $f$ is monotone non-decreasing, differentiable except at the origin and concave. In the general case here $f(x) = \tilde{\theta}(x)$ may not be concave in $[0,x_0]$. However, after a careful inspection of the estimate of $Kf(x)$ in the interval $0<x\leq x_0$, when $f$ is not concave it suffices to make the following changes: in \eqref{eq:Htmid1} replace $f'(x)$ by $\sup_{[x,x_0]} f'$, in \eqref{eq:Htmid4} replace $f'(x/2)$ by $\sup_{[x/2,x]} f'$, replace these terms accordingly in the lower bound \eqref{eq:x>0lb}, and replace $f'(x_0/2)$ in \eqref{eq:estx=x0} by $\sup_{[x_0/2,x_0]} f'$.
	The rest of the proof is exactly the same as in Lemma \ref{lm:Dinifunction}.
\end{proof}

From now on, we always denote the function in Lemma \ref{lm:generalDinifunction} (see \eqref{def:Htilde}) as $\tilde{H}(\cdot)$, which will play the same role as the Heaviside function in Section \ref{sec:Lip}. (In the construction of the function in Lemma \ref{lm:generalDinifunction}, we choose $x_0>0$ sufficiently small, depending on $\theta$, so that 
\eqref{def:x0forinteg-2} holds.)

As in Section \ref{sec:Lip}, we construct a new function $f$ as follows. Let $c$ be a positive real number, and $\{a_k\}$ be a sequence in $\mathbb{R}_+$ such that
\begin{equation}\label{asmp:ak}
	c':= c \sum a_k < \frac{\pi}{2}. 
\end{equation} 
Using the sequence $\{x_k:= 2^{-k}\}$\footnote{In fact, we may take $\{x_k\}$ to be any sequence such that the infinite product $\prod \frac{1}{|x-x_k|^{\frac{ca_k}{\pi}}}$ is integrable near the origin. Or else we may also appeal to the Helson-Szeg\"{o} theorem, which implies that $\exp(-Kf(x))$ is an $A_2$-weight on $\RR$, if $\|f\|_{\infty} < \frac{\pi}{2}$ (see \cite[\S 6.21]{Stein}). Thus in particular $\exp(-Kf(x))$ is locally integrable on $\RR$, for any sequence $\{x_k\}$ as long as \eqref{asmp:ak} holds. 
But for simplicity we just take $x_k = 2^{-k}$ and give a self-contained elementary proof. This assumption is only used in the proof of the claim \eqref{eq:argx=0}.},
we define a function $f: \mathbb{R} \to \mathbb{R}$ as follows
\begin{equation}\label{def:f}
	f(x) = c\sum a_k \tilde{H}(x-x_k). 
\end{equation} 
Clearly $f\in C(\RR)$, $f \in C^1 \left(\RR \setminus \left( \{x_k\}\cup\{0\} \right) \right) $ and
\[ \|f\|_{L^\infty} \leq c\sum a_k = c' < \frac{\pi}{2}. \]
Moreover, we can prove the following lemma:
\begin{lemma}\label{lm:Kfnearxk}
	$Kf(x)$ is well-defined and continuous in $\RR \setminus\left( \{x_k\} \cup \{0\}\right)$. Near each $x_k$ we have that
	\begin{equation}\label{eq:nearxk}
		Kf(x) - ca_k K\tilde{H}(x-x_k) = c\sum_{k' \neq k} a_{k'} K\tilde{H}(x-x_{k'}),
	\end{equation} 
	where the right hand side is continuous and bounded (the bound only depends on $\delta_k$ in \eqref{def:deltak}). In particular
	\[ Kf(x) \to -\infty \quad \text{ as } x\to x_k. \]
\end{lemma}
\begin{proof}
Since
\[ c \sum_{k\leq \ell} a_k \tilde{H} (x-x_k) \to f(x) \quad \text{ in } L^\infty(\RR) \]
and $f \in C_b(\RR)$, we have that
\begin{equation}\label{eq:Kfsum}
	Kf(x) = \lim_{\ell\to +\infty} K \left(c\sum_{k\leq \ell} a_k \tilde{H}(x-x_k) \right) = \lim_{\ell\to +\infty} c\sum_{k\leq \ell} a_k K \tilde{H}(x-x_k), 
\end{equation} 
where the limit is taken in the BMO space and $Kf \in VMO(\RR)$.
On $\RR\setminus\{x_k\}$, for each $k$ the function $K\tilde{H}(\cdot - x_k)$ is pointwise defined. We claim that the limit on the right hand side of \eqref{eq:Kfsum} is well-defined and gives a continuous function on $\RR \setminus\left( \{x_k \} \cup \{0\} \right)$.

Let $x\in \RR \setminus\{0\}$ be an arbitrary point.
(If $x=x_{k}$ for any $k$, we just remove the $k$-th term and consider the summation $\sum_{k'\neq k}$, so we also have that $x\neq x_k$ for every $k$. See \eqref{eq:Kfnearxk-2}.) 
Recall that $K\tilde{H}$ is a continuous function in $\RR \setminus \{0\}$, it is uniformly continuous on compact subsets of $\RR \setminus\{0\}$. Let $E$ be a compact subset of $\RR \setminus\{0\}$ containing $x$, such that $E \cap \{x_k: k\in \NN\} = \emptyset$. We have that for any $y\in E$,
\[ K\tilde{H}(y-x_k) \to K\tilde{H}(y), \quad \text{ as } x_k \to 0, \]
and the convergence is uniform.
In particular, there exists $k_0\in \NN$ depending on $E$, such that for any $k\geq k_0$ and $y\in E$, we have that
\[ \left| K\tilde{H}(y-x_k) \right| \leq \left| K\tilde{H}(y) \right| + 1. \]
Hence for any $m\geq \ell \geq k_0$, we have
\[ c\sum_{k = \ell}^m a_k \left| K\tilde{H}(y-x_k) \right| \leq \left( c\sum_{k=\ell}^m a_k \right) \left( \left| K\tilde{H}(y) \right| + 1 \right)<+\infty. \]
Therefore as an absolutely convergent series of continuous functions,
\[ c\sum a_k K\tilde{H}(y-x_k) = \lim_{\ell \to +\infty} c\sum_{k \leq \ell} a_k K\tilde{H}(y-x_k)  \]
is well-defined and continuous at $x$.

Moreover near each $x_k$, we have
\begin{equation}\label{eq:Kfnearxk-2}
	Kf(x) = c a_k K\tilde{H}(x-x_k) + \lim_{\ell \to +\infty} c \sum_{k' \leq \ell \atop{k' \neq k} } a_{k'} K\tilde{H}(x-x_{k'}). 
\end{equation} 
Since $x_{k'} \neq x_k$ for every $k'\neq k$, we have that
\begin{equation}\label{def:deltak}
	\delta_k := \inf\{ |x_{k'}-x_k|: k'\in \NN \text{ and } k'\neq k \}>0. 
\end{equation} 
Then as long as $|x-x_k|< \delta_k/2$, since 
\[ |x-x_{k'}| \geq |x_k-x_{k'}| - |x-x_k| > \delta_k/2, \]
%
by the estimates of $K\tilde{H}$ away from the origin in Lemmas \ref{lm:Dinifunction} and \ref{lm:generalDinifunction} as well as the assumption \eqref{asmp:ak}, we can show that
\[ Kf(x) - ca_k K\tilde{H}(x-x_k) = \lim_{\ell\to +\infty} c\sum_{k' \leq \ell \atop{k'\neq k}} a_{k'} K\tilde{H}(x-x_{k'}) \]
is well-defined and continuous at $x_k$, as an absolutely convergent series of continuous functions.
In particular,
\[ \text{ as } x\to x_k, \text{ the limit of } Kf(x) - ca_k K\tilde{H}(x-x_k) \text{ exists and is finite}. \]
Therefore $Kf(x) \to -\infty$ as $x\to x_k$ for every $k\in \NN$.
\end{proof}

\textit{Proof of Theorem \ref{thm:main}}.
We can construct a Lipschitz domain $D = \Phi(\RR^2_+)$ using $f, Kf$ as in Section \ref{sec:Lip}. (Again because $K\tilde{H}$ has logarithmic growth at infinity, by \eqref{eq:x<0}, \eqref{eq:estx>1/2} and the analogous estimates in Lemma \ref{lm:generalDinifunction}, the Poisson integral of $Kf$ is well-defined.) Since $f\in C_b(\RR)$, its Poisson integral $V(z)$ converges to $f(x)$ for every $x\in \partial \RR^2_+$; the Poisson integral $W(z)$ of $Kf(x)$ converges to $Kf(x)$ for every $x\in \RR \setminus \left( \{x_k\}\cup \{0\} \right)$, as in the paragraph after \eqref{eq:decayKf}.

Moreover, we claim that for every $(x,t) \in \RR^2_+$, 
\begin{align}
	W(x,t) = P_t \ast Kf(x) & = \lim_{\ell\to+\infty} P_t \ast \left(c \sum_{k\leq \ell} a_k K\tilde{H}(x-x_k) \right) \label{cl:PoissonintegcvC1} \\
	& = \lim_{\ell\to+\infty}  c \sum_{k\leq \ell} a_k P_t \ast K\tilde{H}(x-x_k).\nonumber
\end{align}
(In the last equality, we simply use the linearity of the Poisson integral operator.)
The proof is by studying the Poisson integral in the regions close to $x$ and away from $x$, similar to the proof of the analogous claim \eqref{cl:Poissonintegcv} in Section \ref{sec:Lip}. So we only sketch the key steps here. Let
\[ \delta_0 := \inf\{|(x,t) - (x_k, 0)|: k\in \NN \}>0. \]
For the Poisson integral on the region which is $2\delta_0$-away from $x$ (i.e. the $P_t^1$ term in \eqref{tmp:Ptcloseaway}), we use the lower and upper bounds of $K\tilde{H}$ proven in Lemma \ref{lm:Dinifunction}, the continuity of $K\tilde{H}$ at $x_0$ and $1/2$, combined with \eqref{eq:bdthetader}, to get that
\[ \left|K\tilde{H}(y) \right| \lesssim \log(|y|+1) + \log \frac{1}{\delta_0} + \log \frac{1}{\frac12 - x_0} + \log \frac{1}{x_0}  +  \frac{1}{\delta_0} + Kf(x_0) + Kf\left(\frac12\right), \]
for any $y\in \RR$ with $|y| > \delta_0$. On the other hand, by the estimates \eqref{eq:x<0}, \eqref{eq:x>0lb} and \eqref{eq:x>0ub}, we have that $K\tilde{H}$ is square-integrable near the origin. This can be used to estimate the Poisson integral on the region which is $2\delta_0$-close to $x$ (i.e. the $P_t^2$ term in \eqref{tmp:Ptcloseaway}). This finishes the proof of \eqref{cl:PoissonintegcvC1}.

In particular, as $z=x+it$ converges to $x_k$, by Lemma \ref{lm:Kfnearxk} and the property of the Poisson integral for bounded and continuous functions,
we have that 
\[ W(x,t) - ca_k \cdot P_t \ast \left( K \tilde{H}(x-x_k) \right) \quad \text{ is continuous and bounded}. \]
Combined with the estimates of $K\tilde{H}$ near the origin (see the estimates \eqref{eq:Kfx<0}, \eqref{eq:Htmidt1}, \eqref{eq:Htmidt2},  as well as the definition in \eqref{eq:Htmidt4}), we have
\[ W(x,t) + \frac{ca_k }{\pi} \cdot P_t \ast \left( \int_0^{x_0} \frac{\tilde{\theta}(y)}{y-\cdot} \, dy \right)(x-x_k) \quad \text{ is continuous and bounded near } x_k. \]
Recall that we have shown that the function $x \mapsto \int_0^{x_0} \frac{\tilde{\theta}(y)}{y-x} \, dy$ is well-defined as a principal value and continuous in $\RR \setminus\{0\}$ (in particular, recall that we have proven its continuity at $x_0$ in the proof of Lemma \ref{lm:Dinifunction}). So in order to estimate the Poisson integral
\[ P_t \ast \left( \int_0^{x_0} \frac{\tilde{\theta}(y)}{y-\cdot} \, dy \right)(x) \]
near the origin, let us focus on $\int_0^{x_0} \frac{\tilde{\theta}(y)}{y-x} \, dy$ near the origin. When $x<0$ (and $|x|< \epsilon_0$ for some sufficiently small $\epsilon_0$), we have that
\begin{equation}\label{tmp:logleft0}
	0< \int_0^{x_0} \frac{\tilde{\theta}(y)}{y-x} \, dy \leq \int_0^{x_0} \frac{1}{y-x} \, dy = -\log|x| + \log|x_0-x| < \log \frac{1}{|x|}. 
\end{equation} 
When $x>0$ (and $|x|<\epsilon_0 $), the estimate \eqref{eq:Htmidt3} is not enough for our purpose; instead, we claim that
\begin{equation}\label{cl:theta0x0}
	0 < \int_0^{x_0} \frac{\tilde{\theta}(y) - \tilde{\theta}(x)}{y-x} \, dy \leq C + \frac12 \log \frac{1}{|x|}.
\end{equation}
In fact, it easily follows from \eqref{eq:Htmid3}, \eqref{eq:Htmid4} and \eqref{eq:bdthetader} that
\begin{equation}\label{tmp:thetabeforex}
	0 < \int_0^x \frac{\tilde{\theta}(y) - \tilde{\theta}(x)}{y-x} \, dy \leq C_1 < +\infty. 
\end{equation} 
On the other hand, since $x$ is sufficiently small, we have $2x< x_0$ and hence
\begin{align}
	0 < \int_{x}^{x_0} \frac{\tilde{\theta}(y) - \tilde{\theta}(x)}{y-x} \, dy & = \int_{x}^{2x} \frac{\tilde{\theta}(y) - \tilde{\theta}(x)}{y-x} \, dy + \int_{2x}^{x_0} \frac{\tilde{\theta}(y) - \tilde{\theta}(x)}{y-x} \, dy \nonumber \\
	& \leq \int_x^{2x} \sup_{[x,2x]} \tilde{\theta}' \, dy + \int_{2x}^{x_0} \frac{\tilde{\theta}(y)}{y-x} \, dy \nonumber \\
	& \leq C_2 + \tilde{\theta}(x_0) \left[-\log|x| + \log|x_0-x| \right] \nonumber \\
	& \leq C_2 + \frac12 \log \frac{1}{|x|}.\label{tmp:thetaafterx}
\end{align}
The claim \eqref{cl:theta0x0} then follows by combining \eqref{tmp:thetabeforex} and \eqref{tmp:thetaafterx}.
Since
\[ 0\leq \tilde{\theta}(x) \left(\log|x_0-x| - \log|x| \right) \leq \frac12 \log \frac{1}{|x|}, \]
it follows from \eqref{eq:Htmidt4} that
\begin{equation}\label{tmp:logright0}
	0< \int_0^{x_0} \frac{\tilde{\theta}(y)}{y-x} \, dy = \int_0^{x_0} \frac{\tilde{\theta}(y) - \tilde{\theta}(x)}{y-x} \, dy + \tilde{\theta}(x) \left(\log|x_0-x| - \log|x| \right) \leq C+ \log \frac{1}{|x|}.  
\end{equation} 
Combining \eqref{tmp:logleft0} and \eqref{tmp:logright0}, we have that
\[ 0\leq P_t \ast \left( \mathbbm{1}_{|\cdot|< \epsilon_0} \int_0^{x_0} \frac{\tilde{\theta}(y)}{y-\cdot} \, dy \right)(x) \leq C+ P_t \ast \left(\mathbbm{1}_{|\cdot|< \epsilon_0} \log \frac{1}{|\cdot|} \right)(x). 
\] 
Note that 
\begin{align*}
	P_t \ast \left(\mathbbm{1}_{|\cdot|< \epsilon_0} \log \frac{1}{|\cdot|} \right)(x) & = P_t \ast \left( \log \frac{1}{|\cdot|} \right)(x) - P_t \ast \left(\mathbbm{1}_{|\cdot|\geq \epsilon_0} \log \frac{1}{|\cdot|} \right)(x) \\
	& = \log \frac{1}{|(x,t)|} - P_t \ast \left(\mathbbm{1}_{|\cdot|\geq \epsilon_0} \log \frac{1}{|\cdot|} \right)(x);
\end{align*} 
and $P_t \ast \left(\mathbbm{1}_{|\cdot|\geq \epsilon_0} \log \frac{1}{|\cdot|} \right)(x)$ is bounded when $(x,t)$ is sufficiently close to the origin, since $\mathbbm{1}_{|\cdot|\geq \epsilon_0} \log \frac{1}{|\cdot|}$ is continuous at the origin and thus
\[ P_t \ast \left(\mathbbm{1}_{|\cdot|\geq \epsilon_0} \log \frac{1}{|\cdot|} \right)(x) \to 0 \quad \text{ as } (x,t) \to 0. \]
Therefore when $(x,t)$ is close to the origin,
\begin{align*}
	P_t \ast \left( \int_0^{x_0} \frac{\tilde{\theta}(y)}{y-\cdot}(x) \, dy \right) & = P_t \ast \left(\mathbbm{1}_{|\cdot|<\epsilon_0} \int_0^{x_0} \frac{\tilde{\theta}(y)}{y-\cdot} \, dy \right)(x) + P_t \ast \left( \mathbbm{1}_{|\cdot|\geq\epsilon_0} \int_0^{x_0} \frac{\tilde{\theta}(y)}{y-\cdot} \, dy \right)(x) \\
	& \leq P_t \ast \left(\mathbbm{1}_{|\cdot|<\epsilon_0} \int_0^{x_0} \frac{\tilde{\theta}(y)}{y-\cdot} \, dy \right)(x) + C_3 \\
	& \leq \log \frac{1}{|(x,t)|} + C'.
\end{align*} 
Hence we can use the same argument as in Section \ref{sec:Lip} to show that for any $z, z'$ sufficiently close to $x_k$,
\[ |\Phi(z) - \Phi(z')| \leq \int_{\gamma_{z,z'}} |G(\omega)| = \int_{\gamma_{z,z'}} \exp \left(-W(y,s) \right) \to 0, \quad \text{ as } z,z' \to x_k. \]
That is to say $\Phi(z)$ is continuous at $x_k$.

Next we show that $\Phi(z)$ is also continuous at the origin. To that end, we need to estimate $W(x,t)$ near the origin. 
Since $x_k \to 0$, there exists $k_0 \in \NN$ such that $|x_k| < \epsilon_0/2$ for every $k\geq k_0$.  Then
\begin{align*}
	-W(x,t) & = -\sum ca_k P_t \ast K\tilde{H}(x-x_k) \\
	& = -\sum_{k<k_0} ca_k P_t \ast   K\tilde{H}(x-x_k)  -\sum_{k \geq k_0} ca_k P_t \ast  K\tilde{H}(x-x_k).
\end{align*}
As before,
\begin{align*}
	-\sum_{k \geq k_0} ca_k P_t \ast K\tilde{H}(x-x_k)  & \approx 1 + \sum_{k\geq k_0} \frac{ca_k}{\pi}\cdot  P_t \ast \left(\int_0^{x_0} \frac{\tilde{\theta}(y)}{y-\cdot } \, dy \right)(x-x_k) \\
	& \approx 1 + \sum_{k\geq k_0} \frac{ca_k}{\pi} \cdot P_t \ast \left(\mathbbm{1}_{|\cdot|<\epsilon_0/2}  \int_0^{x_0} \frac{\tilde{\theta}(y)}{y-\cdot} \, dy \right)(x-x_k) \\
	& \leq C''+ \sum_{k\geq k_0} \frac{ca_k}{\pi} \cdot \log \frac{1}{|(x,t)-(x_k,0)|}.
\end{align*}
The constants in the inequality only depends on $c' = c\sum a_k $ and the constants $C_1, C_2$ above; in particular they are independent of $(x,t)$, for $(x,t)$ sufficiently close to the origin.
Again we can use the same argument as in Section \ref{sec:Lip} to show that for any $z, z'$ sufficiently close to the origin,
\[ |\Phi(z) - \Phi(z')| \leq \int_{\gamma_{z,z'}} |G(\omega)| = \int_{\gamma_{z,z'}} \exp \left(-W(y,s) \right) \to 0, \quad \text{ as } z,z' \to 0. \]
That is to say $\Phi(z)$ is continuous at the origin. To sum up, $\Phi: \RR^2_+ \to D$ extends continuously to $\overline{\RR^2_+} \to \overline{D}$. Moreover, by the same argument as in Section \ref{sec:Lip}, it is a homeomorphism.

Note that in proving $\Phi(\infty) = \infty$, we no longer have an explicit formula for $Kf(x)$ in order to estimate the growth of the Poisson integral $W(z)$ of $Kf(x)$ at infinity.
However, by \eqref{eq:x<0} and \eqref{eq:estx>1/2}, we still have that we have that $K\tilde{H}(x)$ grows like $\log|x|$ whenever $x\ll 0$ and $x\gg 1/2$. More precisely, we have that
\begin{equation}\label{def:KtildeHloggrowth}
	K\tilde{H}(x) = K H(x) + h(x) = \frac{1}{\pi} \log|x| + h(x),  
\end{equation} 
where $h(x)$ is a bounded and continuous function away from the origin.  Near the origin $h(x)$ can be written as $h_0(x) + h_+(x)$, with $h_0$ being a continuous and bounded function, and
\[ h_+(x) = \frac{1}{\pi} \int_0^{x_0} \frac{\tilde{\theta}(y)}{x-y} \, dy - \frac{1}{\pi} \log|x| = \frac{1}{\pi} \log\frac{1}{|x|} - \frac{1}{\pi} \int_0^{x_0} \frac{\tilde{\theta}(y)}{y-x} \, dy.  \]
Recall that we have shown before that
\[ 0< \int_0^{x_0} \frac{\tilde{\theta}(y)}{y-x} \, dy \leq \log\frac{1}{|x|} + C, \quad \text{ whenever } |x|< \epsilon_0. \]
Thus
\[ -\frac{C}{\pi} \leq h_+(x) < \frac{1}{\pi} \log \frac{1}{|x|}. \]
Combining the above, we have that
\begin{align*}
	P_t \ast h(x) \leq C + P_t \ast\left(\mathbbm{1}_{|\cdot|< \epsilon_0} h_+(\cdot) \right)(x),
\end{align*}
where
\begin{align*}
	P_t \ast\left(\mathbbm{1}_{|\cdot|< \epsilon_0} h_+(\cdot) \right)(x) & = \frac{1}{\pi} \int_{|y|<\epsilon_0} \frac{t}{(x-y)^2+t^2} h_+(y) \, dy \\
	& \lesssim \left(\int_{|y|<\epsilon_0} \left( \frac{t}{(x-y)^2+t^2} \right)^2 dy \right)^{1/2} \left\|\log|x| \right\|_{L^2([-\epsilon_0, \epsilon_0])}. \\
	& \leq \frac{\sqrt{2\epsilon_0}}{t} \left\|\log|x| \right\|_{L^2([-\epsilon_0, \epsilon_0])}.
\end{align*}
If $|x| \geq 2\epsilon_0$, we have
\begin{align*}
	\left(\int_{|y|<\epsilon_0} \left( \frac{t}{(x-y)^2+t^2} \right)^2 dy \right)^{1/2} \lesssim \sqrt{2\epsilon_0} \cdot \min\left\{ \frac{t}{|x|^2}, \frac{1}{t} \right\};
\end{align*}
if $|x| < 2\epsilon_0$, we have 
\begin{align*}
	\left(\int_{|y|<\epsilon_0} \left( \frac{t}{(x-y)^2+t^2} \right)^2 dy \right)^{1/2} \leq \frac{\sqrt{2\epsilon_0}}{t}.
\end{align*}
In particular, whenever $|(x,t)|> 3\epsilon_0$, there is a uniform lower bound for $-P_t \ast h(x)$ (depending only on $\epsilon_0$).
Combining \eqref{cl:PoissonintegcvC1} and \eqref{def:KtildeHloggrowth}, we get the following lower bound for $(x,t) \in \RR^2_+$ with $|(x,t)| > 3\epsilon_0$:
 \begin{align*}
 	|G(x,t)| = \exp \left(-W(x,t) \right) & = \exp \left(-c \sum a_k P_t \ast  K\tilde{H}(x-x_k) \right) \\
 	& = \exp \left(-\frac{c}{\pi} \sum a_k \log|(x,t)-(x_k,0)|  \right) \cdot \exp \left(-c \sum a_k P_t \ast h(x-x_k)  \right) \\
 	& = \prod \left|(x,t)-(x_k,0)\right|^{-\frac{c}{\pi} a_k}  \cdot  \exp \left(-c \sum a_k P_t \ast h(x-x_k)  \right) \\
 	& \gtrsim (|(x,t)|+1)^{-\frac{c'}{\pi}}
 \end{align*}
 where the constant depends on the uniform lower bound of $-P_t \ast h(x)$  and $c'= c\sum a_k$.
 Therefore by the same argument as in \eqref{tmp:provePhiinfty}, we have that $\Phi(\infty) = \infty$.
 In particular $\partial D = \Phi(\partial \RR^2_+)$, where $\partial D$ denotes the topological boundary of $D$ in $\RR^2$, not the boundary in the Riemann sphere.
 
 
As in Section \ref{sec:Lip}, we know that $G(x)$ is continuous on $\RR \setminus \left( \{x_k\}\cup\{0\}\right)$, and thus
\[ \Phi'(x) \text{ exists and is equal to } G(x). \]
Next, we will show that $\exp(if(x))$ is a unit tangent vector field to $\partial D$ at $p=\Phi(x)$ for every $x\in \RR$ (including when $x=x_k$ and $x=0$). Thus the property $f\in C_b(\RR) \cap C^1(\RR\setminus\left( \{x_k\} \cup \{0\}\right))$ implies that $\partial D$ is $C^1$-regular everywhere, and it is also $C^2$ regular everywhere except at the countably infinite set $\{\Phi(x_k): k\in \NN\}\cup\{\Phi(0)\}$. (At each $x_k$, the modulus of continuity for $f(x)$ is comparable to $\theta(\cdot)$, which fails the Dini condition.) Recall that we relabel the function defined in \eqref{def:Htilde} of Lemma \ref{lm:generalDinifunction} as $\tilde{H}$.
 Hence by the definition of 
 $f(x)$ via $\tilde{H}$ in \eqref{def:f}, it is clear that
\[ f(x) = 0 \text{ when } x < 0, \]
and 
\[ f(x) = c\sum a_k \tilde{H}(x-x_k) = c\sum a_k = c' \text{ when } x \geq  2. \]
Therefore $\partial D$ is flat on $\{\Phi(x): x< 0 \}$ and $\{\Phi(x): x\geq 2\}$.

We claim that for each $k$, $|\Phi'(x_k)| =+ \infty$ and $\arg \Phi'(x_k) = f(x_k)$, in the following sense:
\begin{equation}\label{cl:xk}
	\left| \frac{\Phi(x_k + \epsilon) - \Phi(x_k)}{\epsilon} \right| \to +\infty \text{ and } \arg \frac{\Phi(x_k + \epsilon) - \Phi(x_k)}{\epsilon} \to f(x_k), \quad \text{ as } \epsilon \to 0. 
\end{equation} 
In the above notation, we take the principal branch of the argument function (in fact the argument of $G(z)$ always lies in $(-\pi/2, \pi/2)$, by the bound on $\|f\|_{L^\infty}$).
Recall that $\Phi$ extends continuously to the boundary, and on $\RR \setminus\left(\{ x_k\} \cup\{0\} \right)$ we have
\[ G(x) = \exp\left( -Kf(x) \right) \exp\left( if(x) \right). \]
Therefore
\[ \frac{\Phi(x_k + \epsilon) - \Phi(x_k)}{\epsilon}  = \frac{1}{\epsilon} \int_{x_k}^{x_k+\epsilon} G(x) = \frac{1}{\epsilon} \int_{x_k}^{x_k+\epsilon} \exp\left( -Kf(x) \right) \exp \left( if(x) \right). \]
 By \eqref{eq:nearxk}, when $|\epsilon|$ is sufficiently small we have
\begin{align}
	& \frac{\Phi(x_k + \epsilon) - \Phi(x_k)}{\epsilon} \nonumber \\
	& = \frac{1}{\epsilon} \int_{x_k}^{x_k+\epsilon} \exp\left( -Kf(x) \right) \exp \left( if(x) \right) \nonumber \\
	& = \frac{1}{\epsilon} \int_{x_k}^{x_k+\epsilon} \exp\left( -ca_k K\tilde{H}(x-x_k) \right) \cdot \exp\left( -c\sum_{k' \neq k} a_{k'} K\tilde{H}(x - x_{k'}) \right)  \exp \left( i f(x) \right) \nonumber \\
	& = \exp\left(if(x_k) \right) \exp\left( -c \sum_{k' \neq k} a_{k'} K\tilde{H}(x_k - x_{k'}) \right) \times \nonumber \\
	& \qquad\quad \frac{1}{\epsilon} \int_{x_k}^{x_k+\epsilon} \exp\left( -ca_k K\tilde{H}(x-x_k) \right) \exp\left( \rho(x-x_k) \right)  \exp \left( i \alpha(x-x_k) \right)\label{eq:Phiderxk1}
\end{align}
where we denote the real-valued functions
\[ \rho(\tau) := -c \sum_{k' \neq k} a_{k'} K\tilde{H}(\tau +x_k - x_{k'}) + c \sum_{k' \neq k} a_{k'} K\tilde{H}(x_k - x_{k'}), \]
\[ \alpha(\tau): = f(\tau +x_k) - f(x_k). \]
By the continuity of the functions $\sum_{k'\neq k} a_{k'} K\tilde{H}(x-x_{k'})$ and $f(x)$ at $x_k$, we have that
\begin{equation}\label{eq:cpmod}
	\rho(\tau), \alpha(\tau) \to 0 \text{ as } \tau\to 0. 
\end{equation} 

We will show that $\exp\left( -ca_k K\tilde{H}(x) \right)$ is integrable near the origin\footnote{The Helson-Szeg\"{o} theorem implies that the function $\exp(u(x) + Kv(x))$ is an $A_2$ weight on $\RR$, if $u, v\in L^\infty$ and $\|u\|_\infty < \pi/2$ (see \cite[\S 6.21]{Stein}). In particular, it directly implies that $\exp\left(-ca_k K\tilde{H}(x) \right)$ is locally integrable since $ca_k < \frac{\pi}{2}$. However, in our case there is a more elementary proof of the integrability which also shows \eqref{eq:Phiderxk3}, we present that elementary argument here to make it self-contained.}, and moreover, as $\epsilon \to 0$ ($\epsilon$ can be negative),
\begin{equation}\label{eq:Phiderxk3}
	\frac{1}{\epsilon} \int_{0}^{\epsilon} \exp \left(-c a_k K\tilde{H}(x) \right)\, dx \to +\infty. 
\end{equation}  
Once that is proven, by \eqref{eq:cpmod} and by considering the real part and complex part separately it follows easily that
\begin{align}
	& \dfrac{\frac{1}{\epsilon} \int_{x_k}^{x_k+\epsilon} \exp\left( -ca_k K\tilde{H}(x-x_k) \right) \exp\left( \rho(x-x_k) \right)  \exp \left( i \alpha(x-x_k) \right)}{\frac{1}{\epsilon} \int_{x_k}^{x_k+\epsilon} \exp\left( -ca_k K\tilde{H}(x-x_k) \right) } \nonumber \\
	& = \dfrac{ \int_{x_k}^{x_k+\epsilon} \exp\left( -ca_k K\tilde{H}(x-x_k) \right) \exp\left( \rho(x-x_k) \right)  \exp \left( i \alpha(x-x_k) \right)}{ \int_{x_k}^{x_k+\epsilon} \exp\left( -ca_k K\tilde{H}(x-x_k) \right) } \nonumber \\
	& \to 1 \quad \text{ as } \epsilon \to 0.\label{eq:Phiderxk2}
\end{align}
Therefore combining \eqref{eq:Phiderxk1}, \eqref{eq:Phiderxk2} and \eqref{eq:Phiderxk3}, we conclude the proof of \eqref{cl:xk}.

We first consider the case when $\epsilon<0$. Assume without loss of generality that $|\epsilon| < x_0/2$. By \eqref{eq:Kfx<0} and the definition of $\tilde{H}(x)$ in \eqref{def:Htilde}, for $x<0$ we have\footnote{Here we abuse the notation $\approx$: we write $\approx 1$ to indicate the remainder term is close to some fixed constant, but there is no constant multiple of the term $\int_0^{x_0} \frac{\tilde{\theta}(y)}{y-x} \, dy$ (otherwise there would be a constant multiple of $\frac{ca_k}{\pi}$ in the right hand side of \eqref{eq:expbdx<0}). See also the lower bounds in \eqref{tmp:LBinttheta} and \eqref{tmp:LBinttheta-2}.}
\begin{equation}\label{tmp:integxk}
	- \pi \cdot K\tilde{H}(x) = \int_0^{x_0} \frac{\tilde{\theta}(y)}{y-x} \, dy + \int_{x_0}^{1/2} \frac{g(y)}{y-x} \, dy - \log \left( \frac12 - x \right) \approx \int_0^{x_0} \frac{\tilde{\theta}(y)}{y-x} \, dy + 1. 
\end{equation} 
Let $k_0, N_0, \ell$ be the natural numbers such that 
\begin{equation}\label{def:k0N0ell}
	2^{-k_0} \leq x_0 < 2^{-k_0+1}, \quad -2^{-N_0+1} < \epsilon \leq -2^{-N_0}, \quad -2^{-\ell+1} < x\leq -2^{-\ell}. 
\end{equation} 
The assumption that $|x| \leq |\epsilon| < x_0$ guarantees that $\ell \geq N_0 > k_0$.
We have\footnote{In fact, to show $\exp(-ca_k K\tilde{H}(x) )$ is integrable on $[\epsilon,0]$, it suffices to use the rough estimate $0< \int_0^{x_0} \frac{\tilde{\theta}(y)}{y-x} \, dy \leq 
\log \frac{1}{|x|} $ (see \eqref{tmp:logleft0}), and the fact that $\frac{c}{\pi} a_k <1$. Considering that $\ell \approx \log \frac{1}{|x|}$ and $\tilde{\theta}(r) < 1$ (for sufficiently small $r$), the estimate in \eqref{tmp:UBintheta} is clearly much more precise. We prove \eqref{tmp:UBintheta} here, because it, combined with the lower bound in \eqref{tmp:LBinttheta}, essentially gives us the precise value of the integral $\int_0^{x_0} \frac{\tilde{\theta}(y)}{y-x} \, dy$.}
\begin{align}
	\int_0^{x_0} \frac{\tilde{\theta}(y)}{y-x} \, dy \leq \int_0^{2^{-k_0+1}} \frac{\tilde{\theta}(y)}{y-x} \, dy & = \sum_{i=k_0}^\infty \int_{2^{-i}}^{2^{-i+1}} \frac{\tilde{\theta}(y)}{y-x} \, dy \nonumber \\
	& \leq \sum_{i=k_0}^\infty \tilde\theta(2^{-i+1}) \cdot \frac{2^{-i}}{2^{-i}+2^{-\ell}} \nonumber \\
	& = \sum_{i=k_0}^\infty \tilde\theta(2^{-i+1}) \cdot \frac{1}{1+2^{i-\ell}} \nonumber \\
	& = \sum_{i=k_0}^{\ell} \tilde\theta(2^{-i+1}) \cdot \frac{1}{1+2^{i-\ell}} + \sum_{i=\ell+1}^\infty \tilde\theta(2^{-i+1}) \cdot \frac{1}{1+2^{i-\ell}} \nonumber \\
	& \leq \sum_{i=k_0}^{\ell} \tilde\theta(2^{-i+1}) + \sum_{j=1}^\infty \tilde{\theta}(2^{-j-\ell+1} ) \cdot \frac{1}{1+2^j} \nonumber \\
	& \leq \sum_{i=k_0}^{\ell+1} \tilde\theta(2^{-i+1}).\label{tmp:UBintheta}
\end{align}
Hence \eqref{tmp:integxk} implies that
\begin{equation}\label{eq:expbdx<0}
	\exp\left(-ca_k K\tilde{H}(x) \right) \lesssim \exp \left(\frac{ca_k}{\pi} \sum_{i=k_0}^{\ell+1} \tilde\theta(2^{-i+1}) \right). 
\end{equation} 
Therefore
\begin{align}
	\int_{\epsilon}^0 \exp\left( - ca_k K\tilde{H}(x) \right) dx  & \leq \int_{-2^{-N_0+1}}^0 \exp\left( - ca_k K\tilde{H}(x) \right) dx \nonumber \\
	& \lesssim \sum_{\ell=N_0}^\infty \int_{-2^{-\ell+1}}^{-2^{-\ell}} \exp \left( \frac{ca_k}{\pi} \sum_{i=k_0}^{\ell+1} \tilde\theta(2^{-i+1}) \right) dx \nonumber \\
	& = \sum_{\ell=N_0}^\infty \exp \left( \frac{ca_k}{\pi} \sum_{i=k_0}^{\ell+1} \tilde\theta(2^{-i+1}) \right) \cdot 2^{-\ell} \nonumber \\
	& \lesssim_{a_k, x_0}  \sum_{\ell=N_0}^\infty \exp \left( \frac{ca_k}{\pi} \sum_{i=k_0}^{\ell+1} \left( \tilde\theta(2^{-i+1}) - \frac{\pi}{ca_k} \log 2 \right) \right).\label{tmp:integxk<0}
\end{align}
Since $ca_k< \pi/2$, 
we can choose $x_0$ so that for $\beta \in (0,1)$ fixed,
\begin{equation}\label{def:x0forinteg}
	\theta(8x_0) \leq 2(1-\beta)\log 2 < (1-\beta) \frac{\pi}{ca_k} \log 2\quad \text{ for every } k. 
\end{equation} 
Thus for every $i\geq k_0$, we have that
\[ \tilde\theta(2^{-i+1}) \leq \tilde\theta(2^{-k_0+1}) \leq \tilde{\theta}(2x_0) \leq \theta(8x_0) \leq (1-\beta) \frac{\pi}{ca_k} \log 2, \]
and thus
\begin{align*}
	\int_{\epsilon}^0 \exp\left( - ca_k K\tilde{H}(x) \right) dx \lesssim \sum_{\ell= N_0}^\infty \exp \left(-\ell \cdot\beta \log 2 \right) \lesssim 2^{-\beta N_0} \lesssim |\epsilon|^\beta < +\infty.
\end{align*}
In particular $\exp(-ca_k K\tilde{H}(x))$ is integrable on $[\epsilon,0]$.
On the other hand, as in \eqref{tmp:UBintheta} we can also get a lower bound:
\begin{equation}\label{tmp:LBinttheta}
	\int_0^{x_0} \frac{\tilde{\theta}(y)}{y-x} \, dy \geq \frac14 \sum_{i=k_0+1}^\ell \tilde\theta(2^{-i}). 
\end{equation} 
Hence
\begin{align*}
	\int_{\epsilon}^0 \exp\left( - ca_k K\tilde{H}(x) \right) dx & \geq \int_{-2^{-N_0+1}}^0 \exp\left( -ca_k K\tilde{H}(x) \right) dx \\
	& \gtrsim \sum_{\ell=N_0+1}^\infty  \exp \left( \frac{ca_k}{4\pi} \sum_{i=k_0+1}^\ell \tilde{\theta}(2^{-i}) \right) \cdot 2^{-\ell} \\
	& = \exp \left( \frac{ca_k}{4\pi} \sum_{i=k_0+1}^{N_0} \tilde{\theta}(2^{-i}) \right) \cdot \sum_{\ell = N_0+1}^\infty \exp \left( \frac{ca_k}{4\pi} \sum_{i=N_0+1}^{\ell} \tilde{\theta}(2^{-i}) \right) 2^{-\ell} \\
	& \geq 2^{-N_0} \cdot \exp \left( \frac{ca_k}{4\pi} \sum_{i=k_0+1}^{N_0} \tilde{\theta}(2^{-i}) \right).
\end{align*}
Since $|\epsilon| \approx 2^{-N_0}$, it follows that
\begin{align*}
	\frac{1}{|\epsilon|} \int_{\epsilon}^0 \exp\left( - ca_k K\tilde{H}(x) \right) dx \gtrsim \exp \left( \frac{ca_k}{4\pi} \sum_{i=k_0+1}^{N_0} \tilde{\theta}(2^{-i}) \right).
\end{align*}
Recall that
\begin{align*}
	\sum_{i=k_0+1}^{N_0} \tilde{\theta}(2^{-i}) \geq \sum_{i=k_0+1}^{N_0} \theta(2^{-i}) \geq \sum_{i=k_0+1}^{N_0} \int_{2^{-i-1}}^{2^{-i}} \frac{\theta(r)}{r} dr  = \int_{2^{-N_0-1}}^{2^{-k_0-1}} \frac{\theta(r)}{r} dr \geq \int_{|\epsilon|/2}^{x_0/4} \frac{\theta(r)}{r} dr \to +\infty,
\end{align*}
as $|\epsilon| \to 0$, it follows that
\[ \frac{1}{|\epsilon|} \int_{\epsilon}^0 \exp\left( - ca_k K\tilde{H}(x) \right) dx \to +\infty, \quad \text{ as } \epsilon \to 0-. \]
This finishes the proof of \eqref{eq:Phiderxk3} for the case $\epsilon<0$.

We next consider the case $\epsilon>0$. Assume without loss of generality that $0<\epsilon < x_0/4$. As in the previous case, let $k_0, N_0, \ell$ be the natural numbers such that 
\begin{equation}\label{tmp:dyadicscale}
	2^{-k_0} \leq x_0 < 2^{-k_0+1}, \quad 2^{-N_0} \leq \epsilon < 2^{-N_0+1}, \quad 2^{-\ell} \leq x < 2^{-\ell+1}. 
\end{equation} 
By \eqref{eq:Htmidt1} and \eqref{eq:Htmidt2}, we have that\footnote{We remark that in spite of the equality in \eqref{eq:Htmidt4}, estimating the integral of $\frac{\tilde{\theta}(y) - \tilde{\theta}(x)}{y-x}$ is more convenient than estimating $\int_0^{x_0} \frac{\tilde{\theta}(y)}{y-x} \, dy$ directly, because the latter integral hides some cancellation effect of the integrals before and after $x_0$.}
\begin{align}
	-\pi \cdot K\tilde{H}(x) & \approx 1+ \int_0^{x_0} \frac{\tilde{\theta}(y) - \tilde\theta(x)}{y-x} dy - \tilde\theta(x) \cdot \left(\log x - \log (x_0 - x)  \right) \nonumber \\
	& \approx 1+ \int_0^{x_0} \frac{\tilde{\theta}(y) - \tilde\theta(x)}{y-x} dy + \tilde{\theta}(x) \cdot \log \frac{1}{x}.\label{tmp:integxk>0}
\end{align}
By the estimate \eqref{eq:bdthetader} and the monotonicity of $\tilde\theta(\cdot)$, we have
\begin{align*}
	\int_0^x \frac{\tilde{\theta}(y) - \tilde\theta(x)}{y-x} dy & = \int_0^{x/4} \frac{\tilde{\theta}(x) - \tilde\theta(y)}{x-y} dy + \int_{x/4}^x \frac{\tilde{\theta}(x) - \tilde\theta(y)}{x-y} dy \\ 
	& \leq \int_0^{2^{-\ell-1}} \frac{\tilde{\theta}(x) - \tilde\theta(y)}{x-y} dy + \sup_{[x/4,x]} \tilde\theta' \cdot x \\
	& \lesssim 1 + \sum_{i=\ell+2}^\infty \int_{2^{-i}}^{2^{-i+1}} \frac{\tilde\theta(2^{-\ell+1})  }{2^{-\ell} - 2^{-i+1} } dy \\
	& = 1 + \sum_{i=\ell+2}^\infty  \frac{\tilde\theta(2^{-\ell+1})  }{2^{i-\ell} - 2 } \\
	& \lesssim 1+ \tilde\theta(2^{-\ell+1}),
\end{align*}
and
\begin{align*}
	\int_x^{x_0} \frac{\tilde{\theta}(y) - \tilde\theta(x)}{y-x} dy & = \int_x^{4x} \frac{\tilde{\theta}(y) - \tilde\theta(x)}{y-x} dy + \int_{4x}^{x_0} \frac{\tilde{\theta}(y) - \tilde\theta(x)}{y-x} dy \\
	& \leq \sup_{[x,4x]} \tilde\theta' \cdot 3x + \int_{4x}^{x_0} \frac{\tilde{\theta}(y) }{y-x} dy \\
	& \lesssim \frac{1}{x} \cdot 3x + \int_{2^{-\ell+2}}^{2^{-k_0+1}}\frac{\tilde{\theta}(y) }{y-x} dy \\
	& = 3 + \sum_{i=k_0 }^{\ell-2 } \int_{2^{-i}}^{2^{-i+1}} \frac{\tilde{\theta}(y) }{y-x} dy \\
	& \lesssim 1+ \sum_{i=k_0 }^{\ell-2 } \tilde\theta(2^{-i+1}) \cdot \frac{2^{-i}}{2^{-i}-2^{-\ell+1}} \\
	& \lesssim 1+ \sum_{i=k_0 }^{\ell-2 } \tilde\theta(2^{-i+1}).
\end{align*}
Hence
\begin{equation}\label{tmp:integxk>02}
	\int_0^{x_0} \frac{\tilde{\theta}(y) - \tilde\theta(x)}{y-x} dy \lesssim 1 + \sum_{i=k_0 }^{\ell } \tilde\theta(2^{-i+1}).  
\end{equation} 
Notice that \eqref{tmp:integxk>02} is similar to the estimate in \eqref{tmp:UBintheta} (modulo adding a constant). Hence by a similar argument to \eqref{tmp:integxk<0}, we can show that
\begin{align*}
	\int_0^\epsilon \exp\left(-ca_k K\tilde{H}(x) \right) dx & \lesssim \sum_{\ell =N_0}^\infty 2^{-\ell \left( 1- \frac{ca_k}{\pi} \right)} \exp \left(\frac{ca_k}{\pi} \sum_{i=k_0}^{\ell} \tilde\theta(2^{-i+1}) \right) \\
	& \lesssim \sum_{\ell=N_0}^\infty \exp \left(\frac{ca_k}{\pi} \sum_{i=k_0}^{\ell} \tilde\theta(2^{-i+1}) - \ell \left( 1- \frac{ca_k}{\pi} \right) \log 2 \right).
\end{align*} 
As in \eqref{def:x0forinteg}, for any $\beta \in (0,1)$ fixed, by choosing $x_0$ smaller if necessary, we can guarantee that
\begin{equation}\label{def:x0forinteg-2}
	\theta(8x_0) \leq (1-\beta) \log 2 < (1-\beta) \frac{1-\frac{ca_k}{\pi}}{\frac{ca_k}{\pi}} \log 2, \quad \text{ for every } k. 
\end{equation} 
Therefore it follows that
\begin{equation}\label{eq:integ0e>0}
	\int_0^\epsilon \exp\left(-ca_k K\tilde{H}(x) \right) dx \lesssim \sum_{\ell=N_0}^\infty \exp \left(-\ell\cdot \beta\left( 1- \frac{ca_k}{\pi} \right) \log 2 \right) \lesssim |\epsilon|^{\beta'} < +\infty, 
\end{equation} 
where $0<\beta'<\beta$.

Next we want to show that
\[ \frac{1}{\epsilon} \int_0^\epsilon \exp\left(-ca_k K\tilde{H}(x) \right) dx \gtrsim \frac{1}{\epsilon} \int_0^\epsilon \exp\left(\frac{ca_k}{\pi} \int_{2x}^{x_0} \frac{\tilde\theta(y) - \tilde\theta(x)}{y-x} dy  \right) dx \to +\infty \]
as $\epsilon \to 0+$.
By \eqref{tmp:dyadicscale} and the monotonicity of $\tilde\theta(\cdot)$, we have
\begin{align}
	\int_{2x}^{x_0} \frac{\tilde\theta(y) - \tilde\theta(x)}{y-x} dy \geq \int_{2^{-\ell+2} }^{2^{-k_0}} \frac{\tilde\theta(y) - \tilde\theta(x)}{y-x} dy & = \sum_{i=k_0-1}^{\ell-2} \int_{2^{-i}}^{2^{-i+1}} \frac{\tilde\theta(y) - \tilde\theta(x)}{y-x} dy \nonumber \\
	& \geq \sum_{i=k_0-1}^{\ell-2} \frac{\tilde\theta(2^{-i}) - \tilde\theta(2^{-\ell+1}) }{2^{-i+1} - 2^{-\ell} } \cdot 2^{-i} \nonumber \\
	& \geq \frac12 \sum_{i=k_0-1}^{\ell-2} \left[ \tilde\theta(2^{-i}) - \tilde\theta(2^{-\ell+1}) \right].\label{tmp:LBinttheta-2}
\end{align}
Since $\tilde\theta(\cdot)$ is an increasing function, we remark that $\sum_{i=k_0-1}^{\ell-2} \left[ \tilde\theta(2^{-i}) - \tilde\theta(2^{-\ell+1}) \right]$ increases as $\ell$ increases.
Hence
\begin{align}
	& \frac{1}{\epsilon} \int_0^\epsilon \exp\left(\frac{ca_k}{\pi} \int_{2x}^{x_0} \frac{\tilde\theta(y) - \tilde\theta(x)}{y-x} dy  \right) dx \nonumber \\
	& \quad \geq \frac{1}{\epsilon}  \int_0^{2^{-N_0}} \exp\left(\frac{ca_k}{\pi} \int_{2x}^{x_0} \frac{\tilde\theta(y) - \tilde\theta(x)}{y-x} dy  \right) dx \nonumber \\
	& \quad \gtrsim \frac{1}{\epsilon}  \sum_{\ell=N_0-1}^\infty \int_{2^{-\ell}}^{2^{-\ell+1}} \exp\left( \frac{ca_k}{2\pi} \sum_{i=k_0-1}^{\ell-2} \left[ \tilde\theta(2^{-i}) - \tilde\theta(2^{-\ell+1}) \right] \right) dx \nonumber \\
	& \quad \geq \exp\left(\frac{ca_k}{2\pi} \sum_{i=k_0-1}^{N_0-3} \left[ \tilde\theta(2^{-i}) - \tilde\theta(2^{-N_0+2}) \right] \right) \cdot \frac{1}{\epsilon}  \sum_{\ell=N_0-1}^\infty 2^{-\ell} \nonumber \\
	& \quad \approx \exp\left(\frac{ca_k}{2\pi} \sum_{i=k_0-1}^{N_0-3} \left[ \tilde\theta(2^{-i}) - \tilde\theta(2^{-N_0+2}) \right] \right).\label{tmp:integxklb}
\end{align}
Moreover,
\begin{align}
	\sum_{i=k_0-1}^{N_0-3} \left[ \tilde\theta(2^{-i}) - \tilde\theta(2^{-N_0+2}) \right] \gtrsim \sum_{i=k_0-1}^{N_0-3} \int_{2^{-i+1}}^{2^{-i}} \frac{\tilde\theta(x)- \tilde\theta(2^{-N_0+2})}{x} dx & = \int_{2^{-N_0+2}}^{2^{-k_0+1}}  \frac{\tilde\theta(x)- \tilde\theta(4\epsilon)}{x} dx \nonumber \\
	& \geq \int_{4\epsilon}^{x_0} \frac{\tilde\theta(x)-\tilde\theta(4\epsilon)}{x} dx.\label{tmp:integxklb2}
\end{align}
For each $\epsilon>0$, we denote the positive-valued function $h_\epsilon(x)$ as
\[ h_\epsilon(x) := \frac{\tilde\theta(x)-\tilde\theta(4\epsilon)}{x} \chi_{\{x\geq 4\epsilon\}}. \]
Then for each $x>0$, we have 
\[ h_\epsilon(x) \to \frac{\tilde\theta(x)}{x}, \quad \text{ as } \epsilon \to 0+. \]
Hence Fatou's lemma implies that
\begin{equation}\label{tmp:integxklb3}
	\liminf_{\epsilon \to 0+} \int_0^{x_0} h_\epsilon(x) \, dx \geq \int_0^{x_0} \frac{\tilde{\theta}(x)}{x} \, dx \geq \int_0^{x_0} \frac{\theta(x)}{x} \, dx = +\infty. 
\end{equation} 
Combining \eqref{tmp:integxklb}, \eqref{tmp:integxklb2} and \eqref{tmp:integxklb3}, we conclude that
\[ \frac{1}{\epsilon} \int_0^\epsilon \exp\left(\frac{ca_k}{\pi} \int_{2x}^{x_0} \frac{\tilde\theta(y) - \tilde\theta(x)}{y-x} dy  \right) dx \to +\infty, \quad \text{ as } \epsilon \to 0+. \]

To complete the proof that $\partial D$ is $C^1$-regular, we also claim that
\begin{equation}\label{eq:argx=0}
	\arg \frac{\Phi(\epsilon) - \Phi(0)}{\epsilon} \to f(0) = 0, \quad \text{ as } \epsilon \to 0. 
\end{equation} 
By the definition of $\Phi$, we have
\begin{align}
	\frac{\Phi(\epsilon) - \Phi(0)}{\epsilon} = \frac{1}{\epsilon} \int_0^\epsilon G(x) & = \frac{1}{\epsilon} \int_0^\epsilon \exp\left(-Kf(x) \right)\exp \left( i f(x) \right) \nonumber \\
	& = \exp(if(0)) \cdot \frac{1}{\epsilon} \int_0^\epsilon \exp\left(-Kf(x) \right)\exp \left( i \alpha(x) \right),\label{tmp:derx=0}
\end{align}
where $\alpha(x) := f(x) - f(0) \to 0$ as $|x| \to 0$.
As in \eqref{eq:Phiderxk2}, if $\exp(-Kf(x))$ is integrable on $[0,\epsilon]$, then
\[ \dfrac{\int_0^\epsilon \exp\left(-Kf(x) \right)\exp \left( i \alpha(x) \right)}{\int_0^\epsilon \exp\left(-Kf(x) \right) } \to 1\quad \text{ as } \epsilon \to 0. \]
Combined with \eqref{tmp:derx=0}, this implies \eqref{eq:argx=0}. Hence it suffices to show that
\begin{equation}\label{claim:integx=0}
	\exp\left(-Kf(x) \right) = \exp\left(-c\sum_k a_k K\tilde{H}(x-x_k) \right) \text{ is integrable on } [0,\epsilon]. 
\end{equation}

For $|x| \ll 1$, combining the estimates of $K\tilde{H}(x)$ in \eqref{eq:Kfx<0}, \eqref{tmp:logleft0} (when $x<0$) and in \eqref{eq:Htmidt1}, \eqref{eq:Htmidt2}, \eqref{tmp:logright0} (when $x>0$), we have that
\begin{align*}
	-\pi \cdot K\tilde{H}(x) \leq C + \log \frac{1}{|x|}.
\end{align*}
Hence
\begin{align*}
	\exp\left(-c\sum a_k K\tilde{H}(x-x_k) \right) & \leq \exp\left(C \sum \frac{ca_k}{\pi} \right) \cdot \exp \left( \sum \frac{ca_k}{\pi} \log \frac{1}{|x-x_k|}  \right) \\
	& \lesssim \prod |x-x_k|^{-\frac{ca_k}{\pi}},
\end{align*}
where the constant only depends on the upper bound of $c'$.
Therefore to prove \eqref{claim:integx=0}, it suffices to show that $\prod |x-x_k|^{-\frac{ca_k}{\pi}}$ is integrable near the origin. The latter is indeed the case when we choose $x_k = 2^{-k}$, and we postpone its proof to the appendix.
Therefore the claim \eqref{eq:argx=0} is proven.

Finally, let $\omega$ denote the harmonic measure of $D$ with pole at infinity. As in Section \ref{sec:Lip}, we have that 
\[ d\omega(z) 
= \frac{1}{\left| \Phi'(\Phi^{-1}(z)) \right|} dz, \]
and moreover
\[ \frac{d\omega}{d\sigma}(\Phi(x_k)) = \lim_{\Delta \to \Phi(x_k)} \frac{\omega(\Delta)}{\mathcal{H}^1(\Delta)}
 = 0\quad \text{ for every } k. \]
Therefore
\begin{align*}
	\left\{z\in \partial D: \lim_{\Delta \to z} 
	\frac{\omega(\Delta)}{\mathcal{H}^1(\Delta)} \text{ exists and is equal to }  
	 0  \right\} \setminus \{\Phi(0)\} 
	= 
	\Phi\left(\{x_k: k\in \NN\} \right). 
\end{align*} 
We remark that since $\partial D$ is $C^1$-regular, it is in particular Ahlfors regular, i.e. there are uniform constants $0<C_1\leq C_2$ such that
\[ C_1 r\leq  \mathcal{H}^1(\Delta_r(z)) \leq C_2 r, \quad \text{ for every }z\in \partial D \text{ and } r>0. \] 
Hence
\[ \lim_{\Delta \to z} \frac{\omega(\Delta)}{\mathcal{H}^1(\Delta)} \text{ exists and is equal to } 0 \iff  \lim_{r \to 0} \frac{\omega(\Delta_r(z))}{r} \text{ exists and is equal to } 0.  \]
This finishes the proof of Theorem \ref{thm:main} for the harmonic measure with pole at infinity.

Let $\omega$ (resp. $\omega^\infty$) denote the harmonic measure in $D$ with pole at $X \in D$ (resp. with pole at $\infty$), and let $G(X, \cdot)$ (resp. $G(\infty, \cdot)$) denote the Green's function of the Laplacian in $D$ with pole at $X$ (resp. with pole at $\infty$). By applying the comparison principle in Lemma \ref{lm:comparison} to $G(X, \cdot)$ and $G(\infty, \cdot)$, we have that
\[ G(X, \cdot) \approx G(\infty, \cdot), \]
as long as we are $\dist(X, \partial D)/2$-close to the boundary. By Lemma \ref{lm:CFMS}, it follows that
\[ \omega(\Delta_r(p)) \approx \omega^\infty(\Delta_r(p)), \]
as long as $X \notin B_{2r}(p)$. Therefore
\begin{align*}
	& \left\{ p\in \partial D: \lim_{r\to 0} \frac{\omega(\Delta_r(p))}{r} \text{ exists and is equal to } 0 \right\} \\
	& \qquad = \left\{ p\in \partial D: \lim_{r\to 0} \frac{\omega^\infty(\Delta_r(p))}{r} \text{ exists and is equal to } 0 \right\} \supset \Phi\left(\{ x_k: k\in \NN\} \right). 
\end{align*} 
This finishes the proof of Theorem \ref{thm:main}.\qed

\appendix
\renewcommand{\theequation}{A.\arabic{equation}}
\section{}

\begin{lemma}
	Assume that $\sum b_k < 1/2$ and $x_k = 2^{-k}$. The function	
	\[ g(x) = \prod |x-x_k|^{-b_k} \]
	is integrable near the origin. Moreover,
	\[ \int_{-\epsilon}^\epsilon g(x) \, dx \lesssim \epsilon^{1-\sum b_k}. \]
\end{lemma}	
\begin{proof}
	Let $\epsilon>0$ be sufficiently small. We first prove that
	\[ \int_{-\epsilon}^0 g(x) \, dx \leq 2 \epsilon^{1-\sum b_k} < +\infty. \]
	For each $x\in (-\epsilon, 0)$ and $k\in \NN$, we have
	\[ |x-x_k| = x_k - x \geq |x|. \]
	Hence
	\begin{align*}
		\int_{-\epsilon}^0 g(x) \, dx \leq \int_{-\epsilon}^0 |x|^{-\sum b_k} \, dx = \frac{\epsilon^{1-\sum b_k}}{1-\sum b_k} < +\infty.
	\end{align*}
	To estimate $\int_0^{\epsilon} g(x) dx$, we assume $k_0 \in \NN$ is such that $2^{-k_0} \leq \epsilon < 2^{-k_0+1}$. Then
	\[ \int_0^\epsilon g(x) \, dx \leq \int_0^{2^{-k_0+1}} g(x) \, dx = \sum_{i=k_0 }^\infty \int_{2^{-i}}^{2^{-i+1}} g(x) \, dx. \]
	Let $x\in [2^{-i}, 2^{-i+1}]$ be arbitrary. For every $k\geq i+1$, we have
	\[ |x-x_k| \geq 2^{-i} - 2^{-(i+1)} = 2^{-i-1}; \]
	for every $k\leq i-2$, we have
	\[ |x-x_k| \geq 2^{-(i-2)} - 2^{-i+1} = 2^{-i+1}. \]
	Hence
	\begin{equation}\label{eq:remove}
		\prod_{k\neq i, i-1} |x-x_k|^{-b_k} \leq \left( 2^{-i-1} \right)^{-\sum_{k\geq i+1} b_k} \cdot \left(2^{-i+1} \right)^{-\sum_{k\leq i-2} b_k}. 
	\end{equation} 
	It remains to estimate
	\[ \int_{2^{-i}}^{2^{-i+1}} |x-x_i|^{-b_i} |x-x_{i-1}|^{-b_{i-1}} \, dx. \]
	To that end, let $c_i$ denote the midpoint of the interval $[2^{-i}, 2^{-i+1}]$. Then
	\begin{align}
		\int_{2^{-i}}^{c_i} |x-x_i|^{-b_i} |x-x_{i-1}|^{-b_{i-1}} \, dx & \leq \left( 2^{-i-1} \right)^{-b_{i-1}} \cdot \int_{0}^{2^{-i-1}} t^{-b_i} \, dx \nonumber \\
		& = \left( 2^{-i-1} \right)^{-b_{i-1}} \cdot \frac{(2^{-i-1})^{1-b_i}}{1-b_i}.\label{eq:left}
	\end{align}
	Similarly, 
	\begin{align}
		\int_{c_i}^{2^{-i+1}} |x-x_i|^{-b_i} |x-x_{i-1}|^{-b_{i-1}} \, dx & \leq \left( 2^{-i-1} \right)^{-b_{i}} \cdot \int_{0}^{2^{-i-1}} t^{-b_{i-1}} \, dx \nonumber \\
		& = \left( 2^{-i-1} \right)^{-b_{i}} \cdot \frac{(2^{-i-1})^{1-b_{i-1}}}{1-b_{i-1}}.\label{eq:right}
	\end{align}
	Combining \eqref{eq:remove}, \eqref{eq:left} and \eqref{eq:right}, we obtain
	\begin{align*}
		\int_{2^{-i}}^{2^{-i+1}} g(x) \, dx \leq 4 \cdot (2^{-i-1})^{1-\sum b_k}.
	\end{align*}
	Therefore
	\begin{align*}
		\int_0^\epsilon g(x) \, dx \leq \sum_{i=k_0}^\infty \int_{2^{-i}}^{2^{-i+1}} g(x) \, dx \leq 4 \sum_{i=k_0}^\infty (2^{-i-1})^{1-\sum b_k} \lesssim \epsilon^{1-\sum b_k}.
	\end{align*}
\end{proof}

\end{document}